\documentclass[psamsfonts]{amsart}

\usepackage{amssymb, amsthm, amsmath, amsfonts, mathrsfs}
\usepackage{enumerate}
\usepackage[all,arc]{xy}
\usepackage{hyperref}
\usepackage{graphicx}
\usepackage{overpic}
\usepackage{xcolor}

\newtheorem{thm}{Theorem}[section]

\newtheorem{prop}[thm]{Proposition}
\newtheorem{lem}[thm]{Lemma}
\newtheorem{conj}[thm]{Conjecture}

\theoremstyle{definition}
\newtheorem{defn}[thm]{Definition}

\theoremstyle{remark}
\newtheorem{rem}[thm]{Remark}
\newtheorem{rems}[thm]{Remarks}

\makeatletter
\let\c@equation\c@thm
\makeatother
\numberwithin{equation}{section}

\definecolor{ao}{hsb}{0.67,1,1}
\definecolor{or}{hsb}{0.067,1,1}
\definecolor{green}{hsb}{0.33,1,0.5}

\title{The knot meridians of (1,1)-knot complements are CTF-detected}

\author{Qingfeng Lyu}
\address{Department of Mathematics\\
  Boston College\\
  Chestnut Hill, MA 02467}
\email{lyuqi@bc.edu}

\begin{document}

\begin{abstract}
    For any non-simple (1,1)-knot in $S^3$ or a lens space, we construct a co-oriented taut foliation in its complement that intersects the boundary torus transversely in a suspension foliation of the knot meridian, or the infinity slope. This provides new evidence for a conjecture made by Boyer, Gordon and Hu using slope detections, related to the L-space conjecture. 
\end{abstract}

\maketitle

\section{Introduction}
\label{sec:1}

In~\cite{boyer2017foliations} and~\cite{boyer2021slope}, the idea of slope detections is proposed to investigate the L-space conjecture for toroidal 3-manifolds. In particular, they describe slopes realized by co-orientable taut foliations as ``CTF-detected'':

\begin{defn}[\cite{boyer2021slope}, Definition 5.1]
    Let M be a 3-manifold with $\partial M\cong T^2$. A \textit{boundary slope} [$\alpha$] is an element in the projective space of $H_1(\partial M,\mathbb{R})$. A rational boundary slope $[\alpha]$ is
    \textit{CTF-detected} if there exists a co-oriented taut foliation $\mathcal{F}$ of $M$ intersecting $\partial M$ transversely, such that $\mathcal{F}|_{\partial M}$ is a suspension foliation of the slope [$\alpha$] (that is, a Reebless foliation with a closed leaf of slope $[\alpha]$).
\end{defn}

In this paper we focus on the case when $M$ is a knot complement, and the slope [$\alpha$] is the knot meridian. If the knot comes with a canonical framing, this slope is usually referred to as the \textit{infinity slope}. We will prove the following result:

\begin{thm}
    If $M$ is the complement of a non-simple (1,1)-knot in $S^3$ or some lens space, then the boundary slope of $M$ represented by the knot meridian is CTF-detected.
    \label{thm}
\end{thm}

We give two motivations for investigating CTF-detections of knot meridians.

The first motivation comes from the L-space conjecture. In~\cite{eftekhary2018bordered} it is proved that no toroidal integral homology sphere is an L-space. According to the L-space conjecture, we would expect that all irreducible toroidal integral homology spheres admit co-orientable taut foliations. By the gluing result in~\cite{boyer2021slope} (Theorem 5.2), one could prove this by analysing the CTF-detected slopes of homology solid tori. More precisely they conjectured that 

\begin{conj}[\cite{boyer2021slope}, Conjecture 1.6]
    Let $M\ncong S^1\times D^2$ be an irreducible rational homology solid torus whose longitude $\lambda_M$ is integrally null-homologous, then each rational slope of distance 1 from $\lambda_M$ (any such slope is called \textbf{meridional}) is CTF-detected.
\end{conj}

\cite{boyer2021slope} proved this conjecture for $M$ fibering over the circle. In particular this includes all fibered knot complements. Our result~\ref{thm} provides new evidence for the conjecture in some non-fibered cases, where we only prove CTF-detection for a canonical choice of meridional slope.

In the light of this conjecture, we expect possible generalisations to other meridional slopes, especially when the knot lives in $S^3$. In fact, according to the input from Floer homologies, many of the other meridional slopes (or $1/n$ slopes) yield non L-space surgeries~\cite{ozsvath2010knot}, and according to the L-space conjecture, we should expect the existence of taut foliations of the knot complement intersecting the boundary torus in \textit{product} foliations of these slopes. On the other hand, for our knot meridian or infinity slope, suspension foliations are the best result possible.

Another possible motivation lies in the Floer theory formulation of the Berge's conjecture. To that end, it is conjectured in~\cite{baker2008grid} that

\begin{conj}[\cite{baker2008grid}, Conjecture 1.5]
    For a knot $K\subset L$ living in some lens space $L$, if $K$ is Floer simple, then $K$ is simple.
    \label{conj-s}
\end{conj}

This conjecture is connected to slope detections by the following observation in~\cite{rasmussen2017floer}:

\begin{prop}[\cite{rasmussen2017floer}, discussions below Lemma 3.5]
    Let $L$ be an L-space. A knot $K\subset L$ is Floer simple if and only if the knot meridian is not NLS-detected.
\end{prop}

While the L-space conjecture is still wide open, the only known direction is that CTF implies NLS~\cite{ozsvath2004holomorphic}~\cite{bowden2016approximating}~\cite{kazez2017c0}. A translation of this result in terms of slope detections is also given in~\cite{boyer2021slope}. It is then natural for us to propose that 

\begin{conj}
    For any non-simple knot $K\subset L$ living in some lens space $L$, the knot meridian as the boundary slope of the knot complement is CTF-detected.
\end{conj}

We remark that it is rather easy to show that if $K$ is 1-bridge living in some lens space, then $K$ Floer simple $\Rightarrow$ $K$ simple (\cite{hedden2011floer}). Hence our result is not making actual progress towards Conjecture~\ref{conj-s}. On the other hand, this connection indicates that our result is sharp among (1,1)-knots, in that the knot meridians of simple knots are not CTF-detected.

%% We also remark that for a (1,1)-knot in $S^1\times S^2$ with the knot complement a rational homology solid tori, its knot meridian is actually the longitude.

%The current method is not sensitive to the homology of the total space (in fact, the longitude is not necessarily integrally null-homologous here), and also heavily relies on diagrammatic analysis on the (1,1)-diagrams. Should there be possible generalisations beyond (1,1)-knots, we expect a better control of homology and an algorithm to find nice diagrams to be established.

~\\

\textbf{Organisation of the paper.} In section~\ref{sec:2}, we review the related theories of (1,1)-knots and branched surfaces, prove some technical lemmas, and construct our branched surface for a non-simple (1,1)-knot complement. In section~\ref{sec:3}, we use an inductive argument to show that our branched surface fully carries a lamination. In section~\ref{sec:4}, we turn this lamination into a taut foliation of the (1,1)-knot complement.

~\\

\textbf{Acknowledgements} The author owes many thanks to his advisor Tao Li for countless helpful conversations and tremendous support throughout the writing of this paper. The author thanks the editor and referees for many kind and helpful corrections, comments, and suggestions.

\section{Preparations}
\label{sec:2}

\subsection{(1,1)-diagrams}
\label{subsec:2.1}

We recall some basic properties of (1,1)-knots and (1,1)-diagrams. A (1,1)-knot is a knot that can be represented by some (1,1)-diagram, or \textit{doubly pointed Heegaard diagram of genus one}, denoted as $(\Sigma,\alpha,\beta,z,w)$, where $\Sigma$ represents the torus, $\alpha,\beta$ the curves of the Heegaard diagram and $z,w$ the two basepoints. The knot is recovered by taking properly embedded, boundary-parallel arcs in each solid torus connecting the two basepoints while avoiding the compression disks. For the purpose of this paper we fix an orientation of $\Sigma$ and assume the two curves $\alpha,\beta$ are oriented.

Given $(\Sigma,\alpha,\beta,z,w)$ a doubly pointed Heegaard diagram of genus 1, one can isotope the two curves $(\alpha,\beta)$ into minimal position. Once this is done, it is easy to see that each bigon in the diagram contains at least one basepoint. The resulting (1,1)-diagram is called \textbf{reduced}.

For a reduced (1,1)-diagram, we can use a standard square to represent the torus $\Sigma$, and place the $\alpha$ curve in \textbf{standard position} - the horizontal boundaries of the square (see Figure~\ref{fig:1}). Now $\beta$ is cut into arcs, and we call the arcs connecting the same side of $\alpha$ \textbf{rainbow arcs} and the other arcs connecting different sides \textbf{vertical arcs}. Notice the rainbow arcs are exactly those forming bigons with $\alpha$. For the purpose of this paper, we assume the manifold determined by $(\Sigma,\alpha,\beta)$ is not $S^1\times S^2$, so there must be some vertical arcs.

\begin{figure}[!htb]
    \begin{overpic}{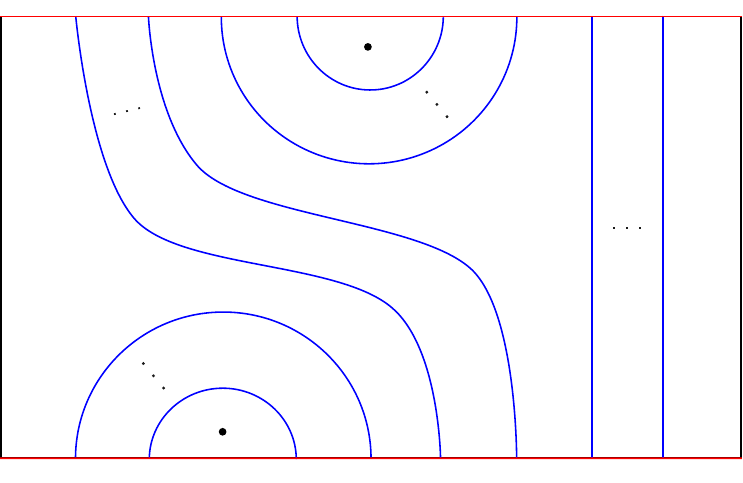}
        \put(27,7){$z$}
        \put(46,59){$w$}
        \put(96,4.5){$\color{red}\alpha$}
        \put(90.5,55){$\color{ao}\beta$}
        \put(9.5,0){$1\;\;\;2\;\;\;...$}
        \put(89,0.5){$p$}
        \put(7.3,64){$s+1\;\;\;...$}
        \put(18,12){$q$}
        \put(18,47){$r$}
        \put(56,49){$q$}
        \put(64.5,40){$(p-2q-r)$}
    \end{overpic}
    \caption{Reduced (1,1)-diagram parametrized by $(p,q,r,s)$}
    \label{fig:1}
\end{figure}
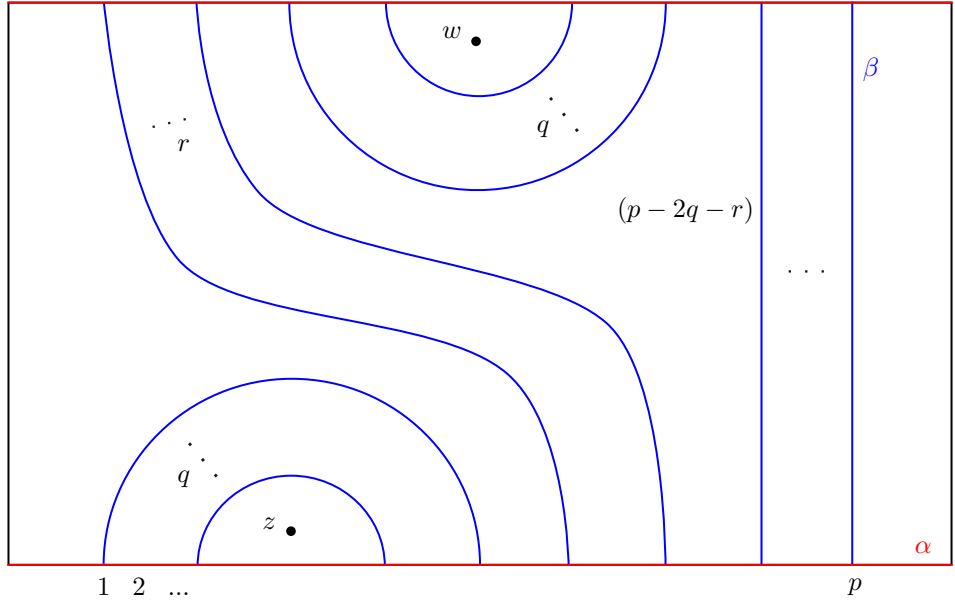

We now claim every bigon contains exactly 1 basepoint. Suppose $\alpha$ is in standard position and there is a rainbow arc bounding two basepoints with $\alpha$. It follows that any bigon on the other side of $\alpha$ cannot bound any basepoint, thus there is no rainbow arc on the other side. But this is impossible, since the number of rainbow arcs on both sides of $\alpha$ should be equal to make the number of $(\alpha,\beta)$-intersection points counted on both sides of $\alpha$ equal.  We also notice that each basepoint can only be bounded by rainbow arcs on the same side of $\alpha$.

It follows that the reduced (1,1)-diagram must look like Figure~\ref{fig:1}. According to the notation of~\cite{rasmussen2005knot}, the reduced diagram is then parametrized by a 4-tuple $(p,q,r,s)$, where $p$ is the number of intersections of $\alpha$ and $\beta$, $q$ is half the number of the rainbow arcs, $r$ is the number of ``leaning'' vertical arcs (as depicted in the picture) and $s$ is the twist occured when identifying the horizontal boundaries $\alpha$, such that the first $\beta$-thread in the ceiling is connected to the $(s+1)$-th $\beta$-thread on the floor.

A reduced (1,1)-diagram is called \textbf{simple} if there are no bigons. In this paper we mainly deal with non-simple (1,1)-diagrams, since simple diagrams would represent simple knots. We first recall the following definition of simple knots:

\begin{defn}[\cite{hedden2011floer}, Definition 1.2]
    A (1,1)-knot represented by some (1,1)-diagram $(\Sigma,\alpha,\beta,z,w)$ is called \textit{simple} if either it bounds a disk, or it can be isotoped to a union of two properly embedded arcs in the two compression disks respectively.
    \label{def:0}
\end{defn}

\begin{lem}
    A simple (1,1)-diagram $(\Sigma,\alpha,\beta,z,w)$ represents a simple knot.
    \label{lem:0}
\end{lem}

\begin{proof}
    We can place $\alpha$ in standard position. Then since there are no bigons, $\beta$ cosists of only vertical arcs. Moreover, $\Sigma$ is cut into quadrilateral pieces by $\alpha$ and $\beta$. If $z,w$ live in the same quadrilateral, then the knot would bound a disk. If not, we can move $z,w$ to the bottom-right corners of the quadrilaterals they live in; notice we can isotope the knot accordingly, such that the two knot arcs connecting $z,w$ still avoid the compression disks except at $z,w$. Moreover, since after our isotopy the knot arc avoiding the $\alpha$-compression disk intersects it on the same side, we can further isotope this knot arc \textit{onto} the compression disk as a properly embedded arc. Similarly, we can isotope the other knot arc onto the $\beta$-compression disk. Hence the knot is simple.
\end{proof}

\begin{figure}[!hbt]
    \begin{overpic}[scale=0.6]{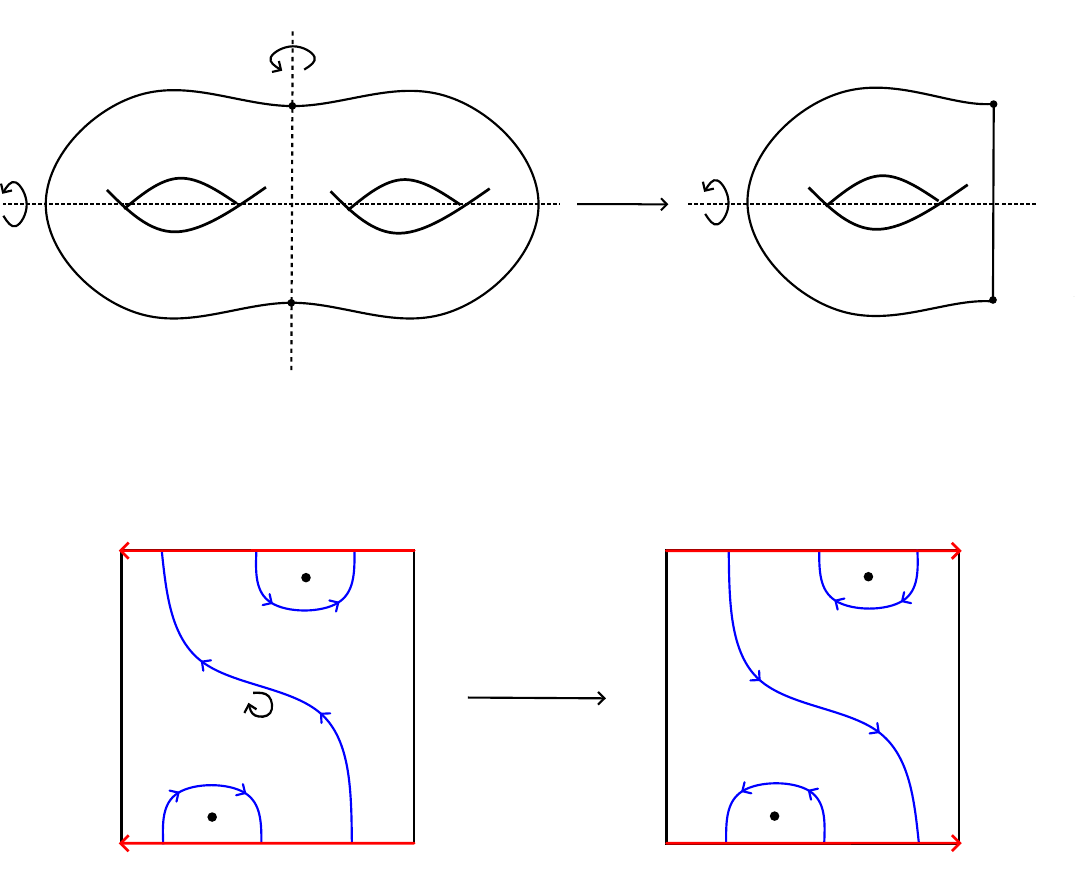}
        \put(50,42){$(a)$}
        \put(50,0){$(b)$}
        \put(25,54){$z$}
        \put(92,53){$z$}
        \put(24.5,70){$w$}
        \put(92,72){$w$}
        \put(2.3,59){$\tilde{h}$}
        \put(67,59){$h$}
        \put(29,77){$\tau$}
        \put(24,13){$h$}
        \put(16,28){$\color{ao}\beta$}
        \put(68,28){$\color{ao} h(\beta)$}
        \put(36,1.5){$\color{red}\alpha$}
        \put(86,0.5){$\color{red}h(\alpha)$}
        \put(17,5){$z$}
        \put(25,28){$w$}
        \put(72,5){$z$}
        \put(81,28){$w$}
    \end{overpic}
    \caption{Hyperelliptic involution}
    \label{fig:2}
\end{figure}
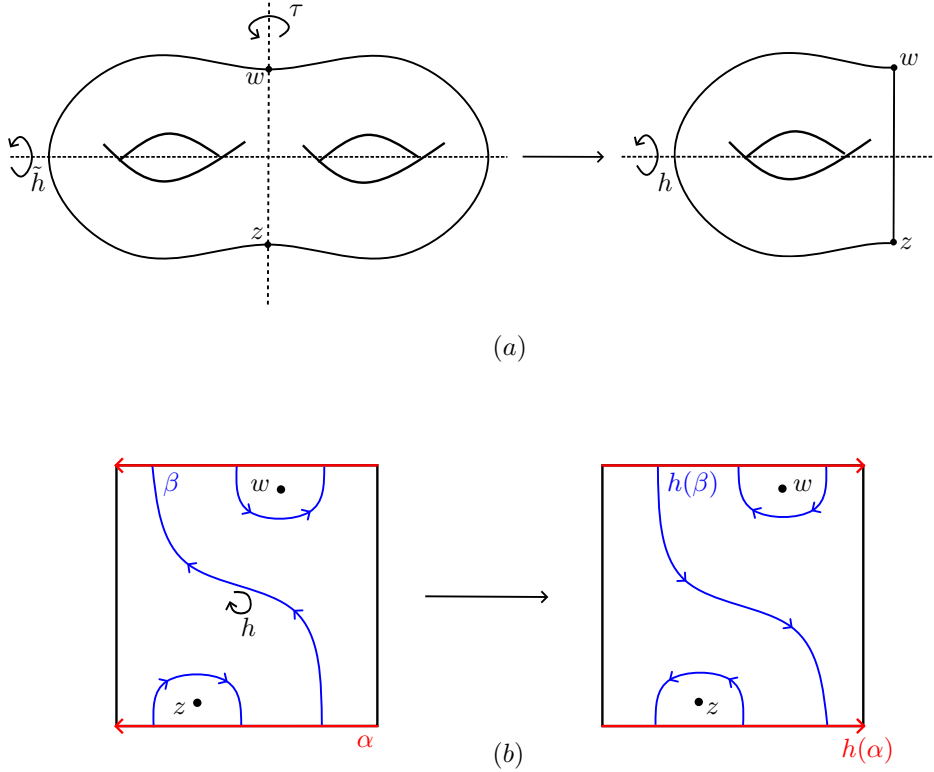

We can regard the torus $\Sigma$ with 2 basepoints $(z,w)$ as a quotient orbifold of a genus-2 surface by a $(\mathbb{Z}/2\mathbb{Z})$-action generated by $\tau$, as shown in Figure~\ref{fig:2}.$(a)$. Then the hyperelliptic involution of the genus-2 surface $\tilde{h}$ descends to $h$ a $\pi$-rotation of the square representing $\Sigma$, exchanging the two basepoints, see Figure~\ref{fig:2}.$(a)(b)$. On the other hand, we know the hyperelliptic involution $\tilde{h}$ fixes all isotopy classes of simple closed curves, reversing orientations of the non-separating ones while preserving those of the separating ones~\cite{haas1989geometry}. Since $\alpha$ and $\beta$ are essential in $\Sigma$, their lifts are non-separating, and hence their orientations are reversed by $h$. Since the diagram obtained by doing $\pi$-rotation is still reduced, we know $(\Sigma,h(\alpha),h(\beta),z,w)$ is actually the same (1,1)-diagram with all orientations reversed. See Figure~\ref{fig:2}.$(b)$. This gives us an important symmetry of the (1,1)-diagrams that we will use later. %For example, the innermost rainbow arcs on both sides appears in the same direction.

We still need some additional analysis on the (1,1)-diagram for later use, which uses the fact that $\alpha,\beta$ are essential simple closed curves. We consider the regions of $\Sigma$ cut off by $\alpha$ and $\beta$. Suppose $\alpha$ is placed in standard position, then for any non-simple (1,1)-diagram (as depicted in Figure~\ref{fig:1}), there are two bigons each containing one basepoint, two hexagons or one octagon (depending on whether there are two kinds or one kind of vertical arcs) bounded by both rainbow arcs and vertical arcs, and many quadrilaterals bounded by either rainbow arcs or vertical arcs. We remark that although we use the standard position of $\alpha$ to identify these regions, the partition into these regions (that there are two bigons and mostly quadrilaterals) does not rely on the position of $\alpha$. From now on when we say bigons in a reduced (1,1)-diagram, we refer to the regions cut out by the $\alpha$ and $\beta$ curves, i.e. only the \textit{innermost} bigons.

\begin{defn}
    Fixing orientations of $(\Sigma,\alpha,\beta)$, for each $\alpha$-arc cut out by $\beta$, we say it is ($\beta$-)\textbf{sink} if it always lies to the left of the $\beta$ curve at its endpoints, ($\beta$-) \textbf{source} if it always lies to the right of the $\beta$ curve at endpoints, and ($\beta$-)\textbf{parallel} otherwise. See Figure~\ref{fig:3} $a\sim c$. It follows immediately that for a quadrilateral region its two boundary $\alpha$-arcs must be of the same type. We call it a ($\beta$-)\textbf{sink sector} if its boundary $\alpha$-arcs are sink, a ($\beta$-)\textbf{source sector} if its boundary $\alpha$-arcs are source, and a ($\beta$-)\textbf{parallel sector} otherwise. See Figure~\ref{fig:3} $d\sim f$. In addition, we call a bigon region a ($\beta$-)\textbf{sink bigon} if its boundary $\alpha$-arc is sink, and a ($\beta$-)\textbf{source bigon} if its boundary $\alpha$-arc is source.
    \label{def:1}
\end{defn}

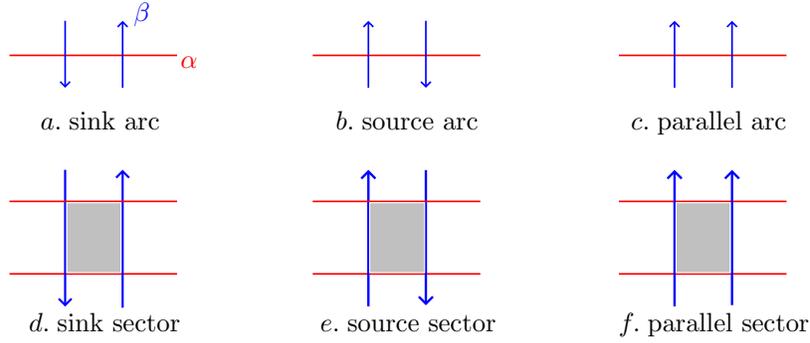
\begin{figure}[!hbt]
    \begin{overpic}[scale=0.6]{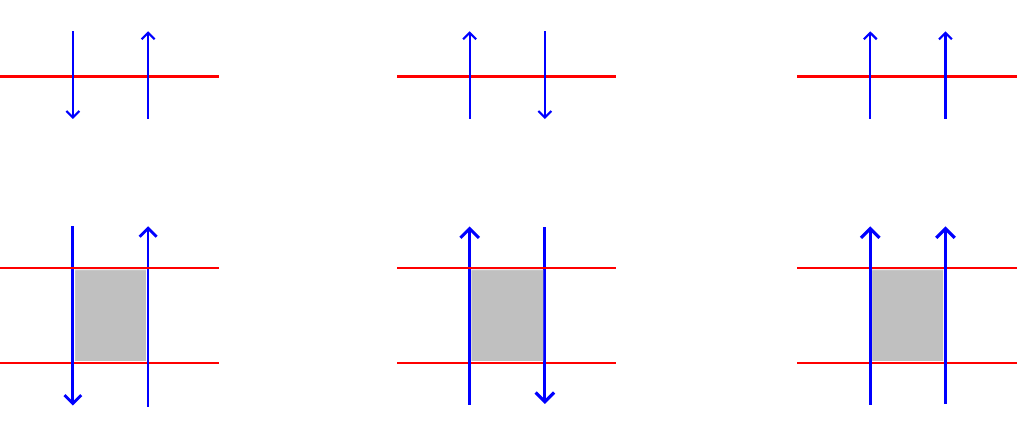}
        \put(22,34){$\color{red}\alpha$}
        \put(16,40){$\color{ao}\beta$}
        \put(4,26){$a.\;$sink arc}
        \put(42,26){$b.\;$source arc}
        \put(80,26){$c.\;$parallel arc}
        \put(2.5,0){$d.\;$sink sector}
        \put(40,0){$e.\;$source sector}
        \put(78.4,0){$f.\;$parallel sector}
    \end{overpic}
    \caption{Arcs and sectors}
    \label{fig:3}
\end{figure}

~\\
\begin{rem}~\
    \begin{enumerate}
        \item One can also define $\alpha$-sink, source, parallel $\beta$-arcs by interchanging $\alpha,\beta$ in the above definition. This allows us to define $\alpha$-sink, source, parallel sectors and $\alpha$-sink, source bigons in the same sense. In later discussions we might omit the $\alpha$- or $\beta$- prefixes, as long as this does not cause confusion in the context.
        \item By the symmetry from the hyperelliptic involution, if a bigon is $\beta$-sink, then the other bigon must be $\beta$-source and vice versa. In fact, the other bigon is the image of the first bigon under hyperelliptic involution, reversing the orientations of the boundary arcs. While the $\pi$-rotation keeps the boundary $\alpha$-arc to be sink, reversing the orientation changes it from sink to source. It follows that there is exactly one $\beta$-sink bigon and one $\beta$-source bigon.
    \end{enumerate}
    \label{rem:1}
\end{rem}

\begin{prop}
    In a non-simple (1,1)-diagram, there is a unique sink $\alpha$-arc and a unique source $\alpha$-arc among the boundary $\alpha$-arcs of the hexagons or octagon.
    \label{lem:1}
\end{prop}

\begin{proof}
    First suppose we have two hexagons, and we look at their $\alpha$-boundary arcs. By enumeration we know there are either 0 parallel $\alpha$-arcs (when the oriented $\beta$-arcs form a ``cycle'') or 2 parallel $\alpha$-arcs (when the $\beta$-arcs form a ``cocycle''), see Figure~\ref{fig:4}. We claim that it must be the second case. In fact, if there is a hexagon with all 3 boundary $\alpha$-arcs sink(resp. source), then by the symmetry from hyperelliptic involution the other hexagon would have all 3 boundary $\alpha$-arcs source(resp. sink). Now since $\beta$ is an essential simple closed curve in $\Sigma$, the two hexagons should be connected by pasting regions along the $\alpha$-arcs. In particular the regions connecting the hexagons should be some (possibly zero) quadrilaterals. However, along quadrilaterals sink $\alpha$-arcs can only connect to sink $\alpha$-arcs by pasting along the sink sectors, so the two hexagons cannot be connected by such pastings. This contradiction implies that each hexagon has 2 parallel boundary $\alpha$-arcs, plus 1 sink arc or 1 source arc. By hyperelliptic involution again, there is exactly 1 sink arc and 1 source arc.

    \begin{figure}[!hbt]
        \begin{overpic}{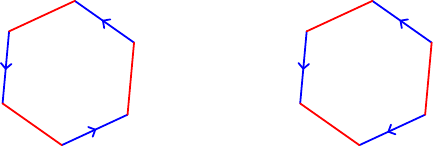}
            \put(11.5,-3){cycle}
            \put(78,-3){cocycle}
        \end{overpic}
        \caption{Hexagons}
        \label{fig:4}
    \end{figure}

    In the case of having one octagon, we claim that the number of its sink $\alpha$-arcs must be odd. In fact, when counting sink $\alpha$-arcs by regions, each sink $\alpha$-arc is counted twice and we should eventually get an even number; on the other hand, the sink sectors contribute 2 each and the single sink bigon contributes 1 (see Remark~\ref{rem:1}), so the octagon should also contribute an odd number. Moreover, by the symmetry from hyperelliptic involution the number of sink and source boundary $\alpha$-arcs of the octagon must be equal, so there must be a unique sink and a unique source boundary $\alpha$-arc.
\end{proof}

The boundary sink $\alpha$-arc of the sink bigon connects the bigon to a sink sector or the sink $\alpha$-arc of the hexagon or octagon. If connected to a sink sector, the other boundary $\alpha$-arc of the sink sector is then connected to another sink sector or the hexagon or octagon. It follows that the sink bigon is eventually connected to the sink $\alpha$-arc of the hexagon or octagon by pasting sink sectors along the boundary $\alpha$-arcs, see Figure~\ref{fig:5}. Similar things happen for the source case.

\begin{defn}
    In a non-simple (1,1)-diagram, the ($\beta$-)\textbf{sink} \textbf{tube} is the union of ($\beta$-)sink sectors that connect the ($\beta$-)sink bigon to the sink boundary $\alpha$-arc of the hexagons or octagon, see Figure~\ref{fig:5}. The ($\beta$-)\textbf{source tube} is the subregion of ($\beta$-)source sectors that connect the ($\beta$-)source bigon to the source boundary $\alpha$-arc of the hexagons or octagon.
    \label{def:2}
\end{defn}

\begin{figure}[!bht]
    \begin{overpic}[scale=0.7]{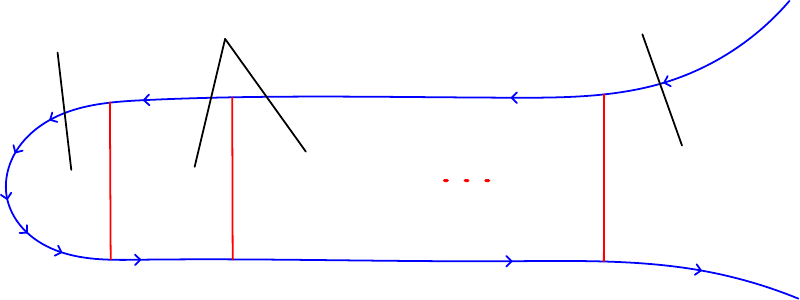}
        \put(-5,33){sink bigon}
        \put(20,34){sink sectors}
        \put(55,35){hexagon or octagon}
    \end{overpic}
    \caption{($\beta$-)sink tube}
    \label{fig:5}
\end{figure}

Notice that by Proposition~\ref{lem:1} the sink and source tubes are well defined for the (1,1)-diagram. It also follows almost immediately from the definition that

\begin{prop}
    The ($\beta$-)sink (resp. source) tube contains all ($\beta$-)sink (resp. source) sectors.
    \label{lem:2}
\end{prop}

\begin{proof}
    For each ($\beta$-)sink sector, along the boundary $\alpha$-arcs it can only connect to another sink sector, the sink bigon, or the sink $\alpha$-arc of the hexagon or octagon. It follows that by pasting sectors along the $\alpha$-arcs it eventually must have one end connecting to the sink bigon and the other end connecting to the hexagon or octagon, thus lies in the sink tube. Similar argument works for the source case.
\end{proof}

\subsection{Branched surfaces}
\label{subsec:2.2}

Our main tool for constructing laminations and foliations is the branched surface. Recall that a \textbf{branched surface} $\mathcal{B}$ is a compact space locally modelled on Figure~\ref{fig:6}.$(a)$. The subset of points that have no neighborhood isomorphic to $\mathbb{R}^2$ is called the \textbf{branch locus} $L(\mathcal{B})$. According to the local picture we can think of the branch locus as a collection of transversely-intersecting immersed curves in $B$. In this paper we sometimes refer to the immersed curves as \textbf{cusp curves}, as they look like cusps in some complementary components of the branched surface. Points where the immersed curves intersect are called \textbf{double points}. Away from the double points, we can define the \textbf{branch direction} on branch locus up to homotopy, s.t. the direction is tangent to the branched surface, transverse to the branch locus, and always points to the side with fewer components, see Figure~\ref{fig:6}.$(a)$. Connected components of $\mathcal{B}-L(\mathcal{B})$ are called \textbf{sectors} of the branched surface. The branched surface is called \textbf{co-oriented} if each sector is orientable, and assigned with an orientation so that these orientations agree at the branch locus, see Figure~\ref{fig:6}.$(a)$.

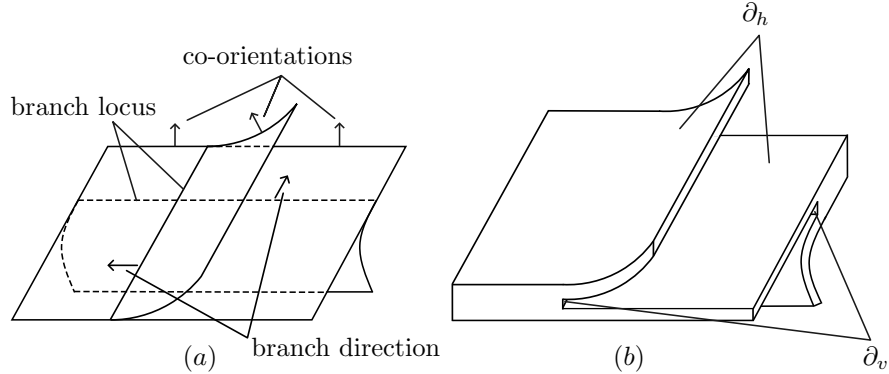
\begin{figure}[!hbt]
    \begin{overpic}[scale=0.6]{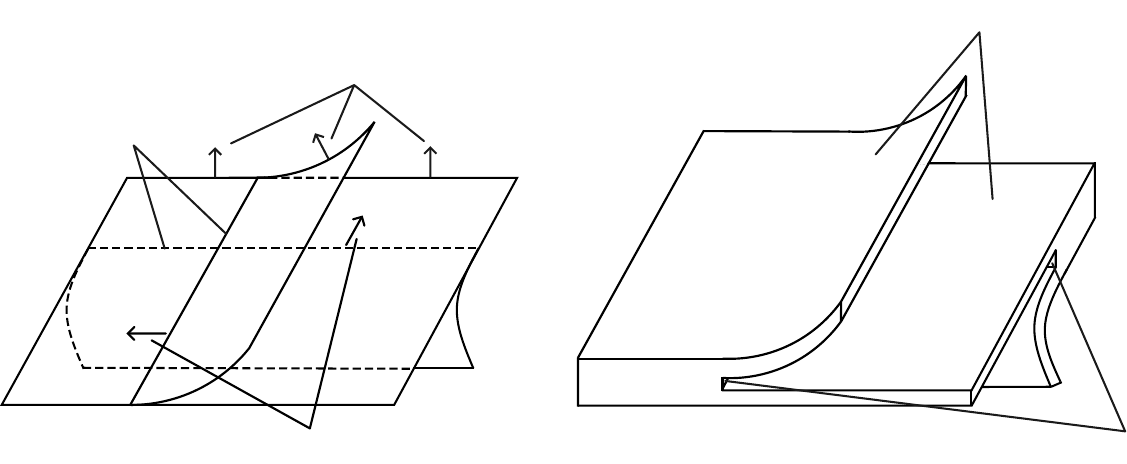}
        \put(0,29){branch locus}
        \put(20,35){co-orientations}
        \put(28,1){branch direction}
        \put(85,40){$\partial_h$}
        \put(99,0){$\partial_v$}
        \put(20,0){$(a)$}
        \put(70,0){$(b)$}
    \end{overpic}
    \caption{Branched surface and its regular neighborhood}
    \label{fig:6}
\end{figure}

For a branched surface $\mathcal{B}$, we can define its regular neighborhood $N(\mathcal{B})$ to be an $I$-bundle over the branched surface. This can be constructed by taking a trivial $I$-bundle for each sector and then pasting them along the branch locus so that $N(\mathcal{B})$ is locally modelled on Figure~\ref{fig:6},$(b)$. We can then define the \textbf{horizontal boundary} $\partial_h N(\mathcal{B})$ to be the boundary of $N(\mathcal{B})$ consisting of the boundaries of the fibers, and the \textbf{vertical boundary} $\partial_v N(\mathcal{B})$ to be the remaining boundary of $N(\mathcal{B})$, which consists of interior segments of $I$-fibers at the branch locus and possibly also fibers transverse to the boundary of the branched surface. For a branched surface with circle boundary, the horizontal boundary of its regular neighborhood is a disjoint union of compact oriented surfaces and the vertical boundary a disjoint union of annuli.

As we regard $N(\mathcal{B})$ as an $I$-bundle over $\mathcal{B}$, there is a bundle projection $\pi:N(\mathcal{B})\rightarrow \mathcal{B}$ collapsing the fibers. Notice that $\pi(\partial_v N(\mathcal{B}))=L(\mathcal{B})\cup \partial \mathcal{B}$.

An important bridge between branched surfaces and foliations would be the laminations. A lamination here is a closed subset of a 3-manifold foliated by (possibly noncompact) properly embedded surfaces. A lamination $\mathcal{L}\subset N(\mathcal{B})$ is said to be \textit{fully carried} by $\mathcal{B}$ if it is transverse to the $I$-fibers and has projection image $\pi(\mathcal{L})=\mathcal{B}$. After possibly replacing these surfaces by (surface $\times I$) and taking (surface $\times \partial I$), we can isotope the lamination so that $\partial_h N(\mathcal{B})\subset \mathcal{L}$ (or equivalently, retract the horizontal boundary onto $\mathcal{L}$ along the fibers). 

Once $\partial_h N(\mathcal{B})\subset \mathcal{L}$ for $\mathcal{B}$ co-oriented, we know that the metric completion $N(\mathcal{B})_{\mathcal{L}}$ of the complement $N(\mathcal{B})-\mathcal{L}$ is a disjoint union of trivial $I$-bundles over oriented surfaces, and thus $\mathcal{L}$ can be trivially extended to a (co-oriented) foliation $\mathcal{F}$ of $N(\mathcal{B})$, transverse to the $I$-fibers.

It follows that to construct a foliation, sometimes it suffices to construct a branched surface that fully carries a lamination. However, usually such foliations may not be taut. Historically,~\cite{gabai1989essential} introduced essential laminations as an analog of taut foliations, and discussed when laminations fully carried by certain branched surfaces are essential. 

\begin{defn}
    A lamination $\mathcal{L}$ in a 3-manifold $M$ is called \textit{essential} if 
    \begin{enumerate}[(i)]
        \item it contains no sphere leaf,
        \item leaves of $\mathcal{L}$ are incompressible and end-incompressible, and
        \item components of $M_{\mathcal{L}}$ (metric completion of $M-\mathcal{L}$) are irreducible.
    \end{enumerate}
    \label{def:3}
\end{defn}

\cite{gabai1989essential} found that, if a branched surface satisfies certain ``incompressibility conditions'', and has no disk of contact (a \textit{disk of contact} is a disk $D\subset N(\mathcal{B})$ transverse to the fibers and such that $\partial D\subset \mathrm{int}(\partial_v N(\mathcal{B})$)), then it only fully carries essential laminations. However, such branched surfaces may not fully carry any lamination. This situation was later improved in~\cite{li2002laminar}, by defining the \textit{laminar branched surfaces}:

\begin{defn}
    A branched surface $\mathcal{B}$ in a closed, oriented 3-manifold $M$ is called \textit{laminar} if 
    \begin{enumerate}[(i)]
        \item $\partial_h N(\mathcal{B})$ is incompressible in $M-int(N(\mathcal{B}))$, no component of $\partial_h N(\mathcal{B})$ is a sphere, and $M-int(N(\mathcal{B}))$ is irreducible (where $int(X)$ is the interior of $X$),
        \item there is no monogon in $M-int(N(\mathcal{B}))$,
        \item there is no Reeb component (i.e. $\mathcal{B}$ does not carry any torus that bounds a solid torus in $M$), and
        \item there is no sink disk (a \textit{sink disk} is a disk branch sector where the branch directions at boundary always point inwards, see Figure~\ref{fig:7}).
    \end{enumerate}
    \label{def:4}
\end{defn}

\begin{figure}[!hbt]
    \begin{overpic}[scale=0.6]{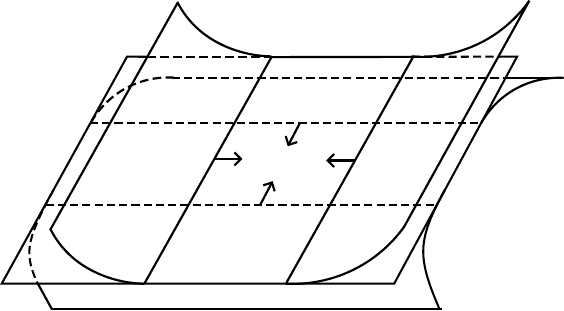}
        
    \end{overpic}
    \caption{Sink disk}
    \label{fig:7}
\end{figure}

Conditions i) to iii) here are exactly the ``incompressibility conditions'' suggested in~\cite{gabai1989essential}, where i) corresponds to incompressibility and ii) corresponds to end-incompressibility. The improvement comes from iv), in that

\begin{thm}[\cite{li2002laminar}, Theorem 1]
    Laminar branched surfaces in closed, oriented 3-manifolds fully carry laminations.
    \label{thm:li}
\end{thm} 

Since laminar branched surfaces have no disk of contact (\cite{li2002laminar}, Lemma 2.2), it follows immediately that the laminations carried are essential. On the other hand, if we do not require the laminations to be essential, we can drop certain ``incompressibility conditions'' of the branched surface. In particular we will use the following lemma:

\begin{lem}
    Let $\mathcal{B}$ be a co-oriented branched surface with boundary a union of circles. If $\mathcal{B}$ is sink disk free, not carrying any torus, and no component of $\partial_h N(\mathcal{B})$ is a disk or a sphere, then $\mathcal{B}$ fully carries a lamination.
    \label{lem:3}
\end{lem}

\begin{proof}
    We can sew in a punctured genus 2 surface to each circular boundary component to get a branched surface $\mathcal{B}'$ without boundary. Consider the regular neighborhood $N(\mathcal{B}')$. Since still no component of horizontal boundary is a disk or a sphere, all components of $\partial N(\mathcal{B}')$ must have Euler characteristic $\leq 0$. It follows that we can cap these boundary components off by pasting some irreducible 3-manifolds with incompressible boundary. Now $\mathcal{B}'$ is laminar in the resulting closed 3-manifold (no monogons since co-oriented), thus fully carries a lamination by Theorem~\ref{thm:li}. $\mathcal{B}$ as a subsurface also fully carries a lamination.
\end{proof}

We also quote a gluing lemma by Gabai here, which is actually used in~\cite{li2002laminar} to prove the main theorem, and will also be used in our proof later.

\begin{lem}[\cite{gabai1992taut}, Operation 2.4.4, modified in \cite{li2002laminar}, Lemma 3.4]
    Let $\mathcal{B}$ be a branched surface with boundary a union of circles. Suppose $\mathcal{B}$ fully carries a lamination without disk leaves. Let $c_1,c_2$ be two boundary components of $\mathcal{B}$. Then the branched surface $\mathcal{B'}$ obtained by gluing $c_1$ and $c_2$ also fully carries a lamination.
    \label{lem:4}
\end{lem}

\subsection{Branched surfaces associated to non-simple (1,1)-diagrams}
\label{subsec:2.3}

Now we define a branched surface associated to a non-simple (1,1)-diagram, which we will eventually prove to fully carry a lamination. The construction is inspired by~\cite{li2022taut}.

\begin{defn}
    Let $(\Sigma,\alpha,\beta,z,w)$ be a non-simple reduced (1,1)-diagram. We orient the two curves so that the boundaries of the two innermost bigons appear clockwise and anti-clockwise (recall we fix an orientation of $\Sigma$). We call the clockwise bigon \textit{source bigon} and the anticlockwise one \textit{sink bigon}. Notice our definition here complies with our earlier definitions of ($\alpha,\beta$-)sink and source bigons in Definition~\ref{def:1}, in the sense that the source bigon is both $\alpha$-source and $\beta$-source, and the sink bigon is both $\alpha$-sink and $\beta$-sink.
    
    We construct a branched surface from the torus $\Sigma$ and the two compression disks bounded by $\alpha,\beta$ on different sides of $\Sigma$ (we call them \textbf{$\alpha$-} and \textbf{$\beta$-disks} respectively). We always assume that the $\alpha$-disk is on the negative side of $\Sigma$ and the $\beta$-disk is on the positive side. Smooth the disks according to orientations so that branch direction points to the left-hand side of the oriented curves (when observing on the oriented torus from the positive side). Then remove the interior of the source bigon and a small open disk in the interior of the sink bigon (one can think of this as puncturing a hole in the sink bigon). The resulting branched surface is called the \textit{branched surface associated to the (1,1)-diagram}. 
    \label{def:5}
\end{defn}

\begin{rem}
    It is a general argument that after removing a source disk (i.e. a disk branch sector where branch directions at boundary always point outwards) we get a new branched surface. Moreover, the branch locus of the new branched surface is the original branch locus \title{minus} the boundary of the source disk. See Figure~\ref{fig:source}.

    \begin{figure}[!hbt]
        \begin{overpic}[scale=0.5]{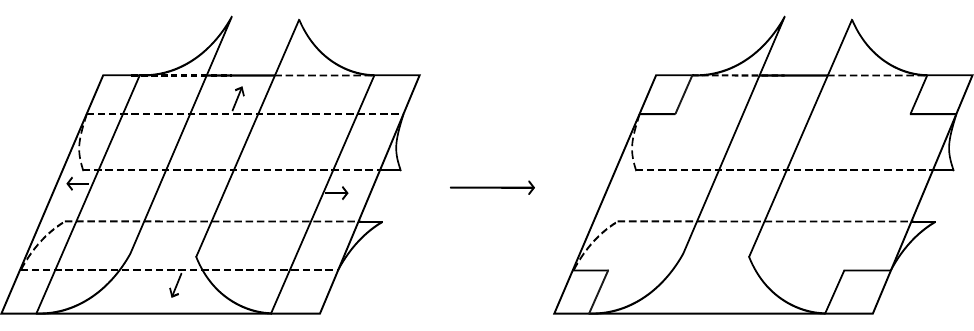}
            
        \end{overpic}
        \caption{Removing a source disk}
        \label{fig:source}
    \end{figure}
\end{rem}

From now on unless otherwise specified, we will assume our reduced (1,1)-diagrams are \textbf{oriented so that there exists an associated branched surface}. Notice that the two possible orientations here for a non-oriented reduced (1,1)-diagram (to have an associated branched surface) differ by the hyperelliptic involution, hence the associated branched surfaces are actually homeomorphic (with homeomorphism induced by the homeomorphism of the diagrams). Hence it makes sense for us to talk about the branched surface associated to a non-simple reduced (1,1)-diagram, without specifying its orientations in advance.

Usually when constructing (taut) foliations using branched surfaces, one needs to extend the foliation from the branched surface neighborhood $N(\mathcal{B})$ to the desired 3-manifold $M$. With our construction this is quite easy:

\begin{lem}
    Let $\mathcal{B}$ be the branched surface associated to a non-simple reduced (1,1)-diagram $(\Sigma,\alpha,\beta,z,w)$. Then its regular neighborhood $N(\mathcal{B})$ is homeomorphic to the corresponding (1,1)-knot complement.
    \label{lem:5}
\end{lem}

\begin{proof}
    Let $L$ be the total lens space or 3-sphere the (1,1)-knot lives in. From the construction of $\mathcal{B}$ there's a canonical way of embedding our branched surface into $L$, identifying $\Sigma$ and the $\alpha,\beta$-disks with the Heegaard torus and compression disks. We notice from the construction that in each of the two solid tori cut out by the Heegaard torus $\Sigma$, the complement of the branched surface is a 3-ball obtained by cutting the solid torus along the compression disk. Moreover, there are two 1-handles connecting these two 3-balls at the removed source bigon and the hole in the sink bigon respectively. It follows that $L-N(\mathcal{B})$ is a solid torus, with its core curve isotopic to the (1,1)-knot in $L$. Hence we can identify $N(\mathcal{B})$ with the (1,1)-knot complement.
\end{proof}

We also prove the following lemma that'll be used in the next section:

\begin{lem}
    The branched surface $\mathcal{B}$ associated to a non-simple (1,1)-diagram contains no disk of contact.
    \label{lem:6}
\end{lem}

\begin{proof}
    First we show that there is only 1 cusp circle and 1 boundary circle in $\mathcal{B}$ (that are the cores of the vertical boundary components of $N(\mathcal{B})$). The boundary circle comes from the hole in the sink bigon. Before removing the source bigon, the branch locus of the model consists of two cusp circles $\alpha$ and $\beta$. After removing the source bigon to get $\mathcal{B}$, the boundary of the source bigon is no longer part of the branch locus. In fact, the branch locus becomes a single cusp circle obtained by connecting the remaining $\alpha$- and $\beta$-arcs at the vertices of the source bigon. Moreover, by isotoping along the compression disks we know on $\partial N(\mathcal{B})$ the new cusp circle (as the core curve of the corresponding vertical boundary annulus) is still isotopic to the boundary of the source bigon. Hence both the cusp circle and the boundary circle represent the knot meridian. Since the meridian of a knot cannot bound disks in the knot complement (otherwise we will be getting a non-separating 2-sphere in $L$), there is no disk of contact in $\mathcal{B}$. 
\end{proof}

\begin{rem}
    We can also regard $N(\mathcal{B})$ as properly embedded in the corresponding (1,1)-knot complement $M$, such that $N(\mathcal{B})$ intersects $\partial M\cong T^2$ in the vertical annulus transverse to the boundary circle of $\mathcal{B}$ (in the sink bigon). $M-N(\mathcal{B})$ is then homeomorphic to $\mathring{A}\times (0,1)$, where $A$ is an annulus (see Figure~\ref{fig:28} bottom-right, where $A\times \{0\}$ and $A\times \{1\}$ are the two annular components of $\partial_h N(\mathcal{B})$, and $\partial A\times [0,1]$ consist of the vertical annulus at the cusp circle, and an annulus on $\partial M$). In the following sections we will first construct co-oriented foliations for $N(\mathcal{B})$, and then trivially extend the foliation to $M$ via this proper embedding.
\end{rem}

\section{Constructing laminations from (1,1)-diagrams}
\label{sec:3}

In this section we use an inductive argument to prove the following statement:

\begin{prop}
    Every branched surface associated to a non-simple reduced (1,1)-diagram fully carries a lamination.
    \label{prop:1}
\end{prop}

The outline to prove Proposition~\ref{prop:1} is as follows. In subsection~\ref{subsec:3.1} we define primitive (1,1)-diagrams, and claim that their associated branched surfaces fully carry laminations. In subsection~\ref{subsec:3.2} we describe a reduction operation of (1,1)-diagrams, and prove that there's a finite hierarchy to reduce any non-simple, reduced (1,1)-diagram to a primitive one. In subsection~\ref{subsec:3.3} we illustrate that laminations are ``carried over'' in the hierarchy. (That is, if the associated branched surface of the (1,1)-diagram after the reduction operation in subsection~\ref{subsec:3.2} fully carries a lamination, then the associated branched surface of the (1,1)-diagram before reduction also fully carries a lamination.) A proof for Proposition~\ref{prop:1} is thus given at the end of subsection~\ref{subsec:3.3}, minus technical details. In subsections~\ref{subsec:3.4} and~\ref{subsec:3.5} we deal with the technical details and complete all the proofs.

\subsection{Primitive (1,1)-diagrams}
\label{subsec:3.1}

First we define primitive (1,1)-diagrams, and claim that their associated branched surfaces fully carry laminations. (We note that ``primitive'' here is just describing the diagram and doesn't imply known properties of the knot.) Recall the $\beta$-sink/source/parallel $\alpha$-arcs we defined in Definition~\ref{def:1}, and also the $\alpha$-sink/source/parallel $\beta$-arcs in Remark~\ref{rem:1} following it.

\begin{defn}
    Let $(\Sigma,\alpha,\beta,z,w)$ be a reduced (1,1)-diagram. The diagram is called \textit{primitive} if  
    \begin{enumerate}[(i)]
        \item any quadrilateral has ($\beta$-)parallel $\alpha$-arcs if and only if it has ($\alpha$-)parallel $\beta$-arcs, and
        \item if there are two hexagons, then for each hexagon the two non-parallel boundary arcs (one $\alpha$-arc and one $\beta$-arc) are next to each other.
    \end{enumerate}
    \label{def:6}
\end{defn}

When one of the curves is placed in standard position (recall our definition at the start of subsection~\ref{subsec:2.1}, or see the $\alpha$-curve in Figure~\ref{fig:1}), there is a more direct description of being primitive. In fact, we have:

\begin{prop}
    Let $(\Sigma,\alpha,\beta,z,w)$ be a reduced (1,1)-diagram where $\alpha$ is placed in standard position. Then the diagram is primitive if and only if
    \begin{enumerate}[(i)]
        \item its rainbow arcs are in alternating direction, and
        \item its vertical arcs are in the same direction, see Figure~\ref{fig:8}.
    \end{enumerate}
    \label{lem:7}
\end{prop}

\begin{figure}[!hbt]
    \begin{overpic}[scale=0.8]{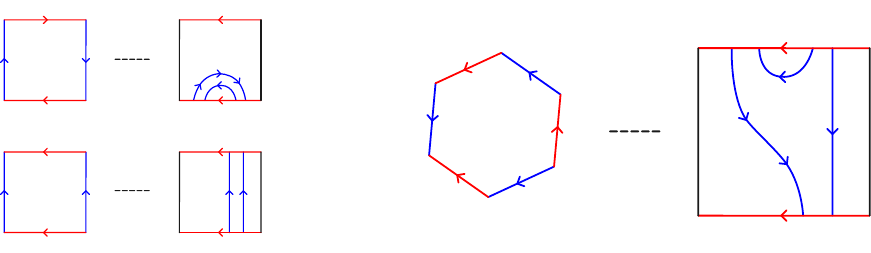}
        \put(97,2){$\color{red} \alpha$}
        \put(96,16){$\color{ao} \beta$}
    \end{overpic}
    \caption{Quadrilaterals and Hexagons in primitive (1,1)-diagrams}
    \label{fig:8}
\end{figure}

\begin{proof}
    Suppose the diagram is primitive and $\alpha$ is in standard position. Quadrilaterals bounded by $\alpha$ and the rainbow arcs then have non-($\alpha$-)parallel $\beta$-arcs, hence must have non-($\beta$-)parallel $\alpha$-arcs, as depicted in Figure~\ref{fig:8} upper-left. It follows that the rainbow arcs must be alternating. On the other hand, quadrilaterals bounded by $\alpha$ and vertical arcs have parallel $\beta$-arcs, hence also have parallel $\alpha$-arcs, indicating that each kind of vertical arcs must be in the same direction (see Figure~\ref{fig:8} bottom-left). It remains to show that if there are two kinds of vertical arcs, they also go in the same direction. Notice since there are two kinds of vertical arcs, there must be two hexagons. For each hexagon, the non-($\alpha$-)parallel $\beta$-arc must be the rainbow arc, see Figure~\ref{fig:8} right. Now since the non-($\beta$-)parallel $\alpha$-arc is next to the rainbow $\beta$-arc, the opposite $\alpha$-arc in between the vertical $\beta$-arcs must be parallel, indicating that the two kinds of vertical arcs are in the same direction. 
    
    One can check that these translations also work for the inverse direction of the statement. The proposition then follows.  
\end{proof}

We remark that by our Definition~\ref{def:6} the diagram being primitive is independent of the position of the diagram, i.e. if we place $\beta$ in standard position instead, we will still get alternating rainbow $\alpha$-arcs and vertical $\alpha$-arcs in the same direction.

\begin{prop}
    Branched surfaces associated to primitive (1,1)-diagrams fully carry laminations.
    \label{prop:2}
\end{prop}

We defer the proof of this proposition to subsection~\ref{subsec:3.4}, where we introduce the ``sink tube push'' to eliminate the sink disks in the branched surface.

\subsection{Reducing (1,1)-diagrams to primitive}
\label{subsec:3.2}

In this subsection we introduce a reduction operation on reduced (1,1)-diagrams that is crucial to our proof. We first make the following observation of (1,1)-diagrams:

\begin{defn}
    Let $(\Sigma,\alpha,\beta,z,w)$ be a reduced (1,1)-diagram. An essential curve $\gamma$ is said to carry $\alpha$ if 
    \begin{enumerate}[(i)]
        \item it is disjoint from and isotopic to $\alpha$ in $\Sigma$ (without the marked points $z,w$),
        \item it only intersects the $\alpha$-parallel $\beta$-arcs.
        \item $(\Sigma,\gamma,\beta,z,w)$ is also a reduced (1,1)-diagram.
    \end{enumerate}
    \label{def:7}
\end{defn}

\begin{figure}[!htb]
    \begin{overpic}[scale=0.47]{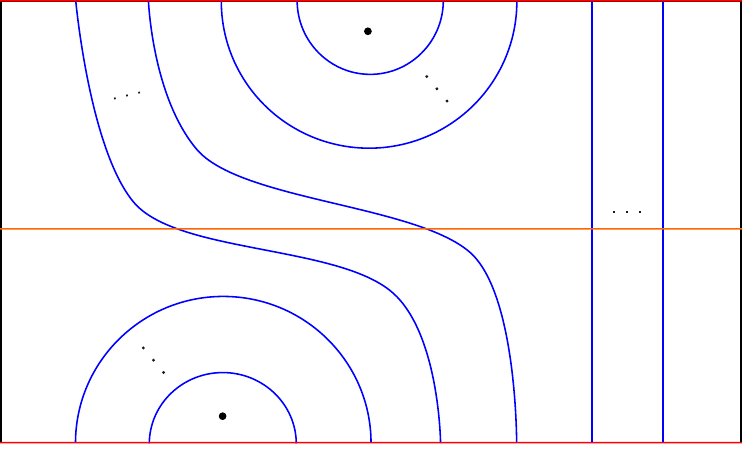}
        \put(95,34){$\color{or} \gamma$}
        \put(95,4){$\color{red} \alpha$}
        \put(91,50){$\color{ao} \beta$}
    \end{overpic}
    \caption{Carrying curve}
    \label{fig:9}
\end{figure}

\begin{rem}
    When $\alpha$ is placed in standard position, condition (ii) is equivalent to saying that $\gamma$ only intersects vertical arcs, see Figure~\ref{fig:9}.
    \label{rem:2}
\end{rem}

\begin{prop}
    Let $(\Sigma,\alpha,\beta,z,w)$ be a reduced (1,1)-diagram. Then there exists $\gamma$ carrying $\alpha$ and it is unique up to isotopy in the twice-punctured torus $(\Sigma,z,w)$.
    \label{lem:8}
\end{prop}

\begin{proof}
    We can place $\alpha$ in standard position. Then we claim the curve $\gamma$ as depicted in Figure~\ref{fig:9} is carrying $\alpha$. Conditions (i) and (ii) are easily checked. For (iii), notice that bigons bounded by $\gamma$ and $\beta$ must intersect $\alpha$, hence contain some bigon bounded by $\alpha$ and $\beta$, hence must also contain some basepoint. It follows that the new (1,1)-diagram $(\Sigma,\gamma,\beta,z,w)$ is also reduced.

    For uniqueness, we first claim that each region cut off by $\alpha$ and $\beta$ has 0 or 2 $\alpha$-parallel $\beta$-arcs. In fact, one can directly check that bigons have 0, $\alpha$-parallel quadrilaterals have 2, other quadrilaterals have 0, and the hexagons or octagon have 2 (recall Proposition~\ref{lem:1}). Now since $\gamma$ is essential, it cannot be contained in a single region. Since it is disjoint from $\alpha$, it can only exit and enter regions at the $\beta$-arcs. Since it is closed, when it enters a region it must also exit it. Since the new (1,1)-diagram is reduced, it cannot exit at the same arc as entry (otherwise there will be a bigon with no basepoint). It follows that $\gamma$ is characterized by a cycle of regions with ($\alpha$-)parallel $\beta$-arcs that it passes through, see Figure~\ref{fig:10}. Now place $\alpha$ in standard position. As can be seen in Figure~\ref{fig:9}, there is a unique such cycle including all the regions with parallel $\beta$-arcs. This determines the unique (up to isotopy) curve $\gamma$ carrying $\alpha$.
    \begin{figure}[!hbt]
        \begin{overpic}[scale=0.7]{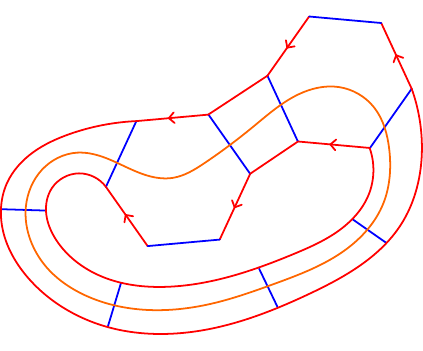}
            \put(40,35){$\color{or} \gamma$}
            \put(40,58){$\color{red} \alpha$}
            \put(80,80){$\color{ao} \beta$}
        \end{overpic}
        \caption{A cycle formed by regions with parallel $\beta$-arcs}
        \label{fig:10}
    \end{figure}
\end{proof}

\begin{rem}
    Since the unique cycle formed by all regions with ($\alpha$-)parallel arcs characterizes the carrying curve (of $\alpha$) up to isotopy, we can use this to find the carrying curve without having the original curve placed in standard position, see the $\delta$-curve carrying $\beta$ in Figure~\ref{fig:12}.
    \label{rem:3}
\end{rem}

Our plan is to replace the original curve by the carrying curve to get a ``simpler'' diagram. However, only non-simple (1,1)-diagrams have associated branched surfaces. Hence we need to guarantee that the new diagram is still non-simple. This happens exactly when the original (1,1)-diagram is not primitive.

\begin{prop}
    If a non-simple reduced (1,1)-diagram $(\Sigma,\alpha,\beta,z,w)$ is not primitive, with the curves $\gamma,\delta$ carrying $\alpha,\beta$ respectively, then either $(\Sigma,\gamma,\beta,z,w)$ or $(\Sigma,\alpha,\delta,z,w)$ is non-simple.
    \label{lem:9}
\end{prop}

\begin{proof}
    Suppose $\alpha$ is placed in standard position. Since the diagram is not primitive, there are either rainbow arcs bounding $\beta$-parallel sectors, or vertical arcs in different directions. In the first case, by Remark~\ref{rem:3} the $\delta$ curve passes through the $\beta$-parallel sector bounded by rainbow arcs (see Figure~\ref{fig:12} for an example). It follows that the $\delta$-arc in this sector is a rainbow arc in $(\Sigma,\alpha,\delta,z,w)$, indicating non-simpleness. In the second case, we first notice that as in Figure~\ref{fig:9} the $\gamma$ curve intersects all vertical arcs. Since the vertical arcs are not in the same direction, there exists some non-$\beta$-parallel $\gamma$-arc. Now if we place $\beta$ in standard position instead, this $\gamma$-arc must be a rainbow arc (vertical arcs are parallel). It then follows that $(\Sigma,\gamma,\beta,z,w)$ is non-simple in this case.
\end{proof}

We now show that, starting with any non-simple reduced (1,1)-diagram, we will get a primitive diagram after finitely many such replacements.

\begin{prop}
    Let $G=(\Sigma,\alpha,\beta,z,w)$ be a non-simple reduced (1,1)-diagram. Then there exists $n\in \mathbb{N}$ and a finite hierarchy \[G=G_0,G_1,...,G_n,\] where $G_i=(\Sigma,\alpha_i,\beta_i,z,w)$ are non-simple reduced (1,1)-diagrams, such that for each $i<n$, $(\alpha_{i+1},\beta_{i+1})$ is obtained by replacing one of $(\alpha_i,\beta_i)$ by its carrying curve in $G_i$, and $G_n$ is primitive.
    \label{prop:hierarchy}
\end{prop}

\begin{proof}
    By Proposition~\ref{lem:9}, we can always do such replacements for a non-simple, non-primitive, reduced (1,1)-diagram. Hence we only need to show that this must stop in finitely many steps. Suppose $G_i=(\Sigma,\alpha_i,\beta_i,z,w)$ is non-primitive, and $G_{i+1}=(\Sigma,\alpha_{i+1},\beta_{i+1},z,w)$ is obtained by replacing $\alpha_i$ with its carrying curve $\alpha_{i+1}$ while $\beta_{i+1}=\beta_i$. Now if we place $\alpha_i$ in standard position when viewing $G_i$, we can spot $\alpha_{i+1}$ as in Figure~\ref{fig:9}, in that it only intersects vertical $\beta_i$-arcs. Since $G_i$ is non-primitive, there are rainbow arcs. Hence the number of intersections of $(\alpha_{i+1},\beta_{i+1})$ is strictly less than that of $(\alpha_i,\beta_i)$. This suffices for a finiteness argument.
\end{proof}

\subsection{Inductive construction of laminations}
\label{subsec:3.3}

In this subsection we set up the proof of the following proposition:

\begin{prop}
    Let $(\Sigma,\alpha,\beta,z,w)$ be a reduced (1,1)-diagram that is not primitive, with $\delta$ carrying $\beta$. Suppose $(\Sigma,\alpha,\delta,z,w)$ is non-simple and its associated branched surface fully carries a lamination. Then the branched surface associated to $(\Sigma,\alpha,\beta,z,w)$ also fully carries a lamination.
    \label{prop:3}
\end{prop}

The proof of this proposition relies on a series of constructions of branched surfaces. We will introduce these constructions in several steps. Each step is followed by one or two lemma(s) showing that the ``fully carrying laminations'' property is carried over. The convinced readers might skip the proofs first to get an overview of the whole construction. 
~\\

\textbf{Step 1: Attaching the $\delta$-annulus.} Let $\mathcal{B}$ be the branched surface associated to $(\Sigma,\alpha,\beta,z,w)$. Consider the curve $\delta$ on the torus. Pick an orientation of $\delta$ so that the pair $(\alpha,\delta)$ forms sink and source bigons. Now we can attach an annulus to $\mathcal{B}$ by identifying one of its boundary components with the $\delta$ curve, so that it is attached to $\Sigma$ on the same side as the $\beta$-disk, and then smooth it according to the orientation, so that when observed on $\Sigma$ the branch direction still points to the left of the $\delta$-curve direction. We call the resulting branched surface $\mathcal{B'}$. (We remark that we attach an annulus (rather than a disk) because the $\delta$-disk is supposed to intersect the (1,1)-knot represented by $(\Sigma,\alpha,\beta,z,w)$ essentially once, see Figure~\ref{fig:14}.)

\begin{lem}
    If $\mathcal{B'}$ fully carries a lamination, then so does $\mathcal{B}$.
    \label{lem:11}
\end{lem}

\begin{proof}[Proof of Lemma~\ref{lem:11}]
    By Proposition~\ref{lem:5} we can identify $N(\mathcal{B})$ with the knot complement. Recall from the proof of Lemma~\ref{lem:6} that both the cusp circle and the boundary circle of $\mathcal{B}$ represent the knot meridian on $\partial N(\mathcal{B})$. We claim that the $\delta$ curve on $\partial N(\mathcal{B})$ (where the $\delta$-annulus is to be attached) also represents the knot meridian. In fact, we can embed $\mathcal{B}$ in the total space $L$ canonically as described in the proof of Proposition~\ref{lem:5}, where the torus $\Sigma$ serves as a Heegaard surface of $L$. Now since $\delta$ is essential and disjoint from $\beta$, it is also a compression curve (of the Heegaard splitting of $L$), and thus bounds a disk in $L-N(\mathcal{B})$. However, since $\delta$ intersects $\alpha$ in strictly fewer points than $\beta$ (see Figure~\ref{fig:9}, where the carrying curve only intersects the vertical arcs), $\delta$ and $\beta$ are not isotopic on the twice-punctured torus $(\Sigma,z,w)$. It follows that on $\Sigma$ the two annuli cut off by $\beta$ and $\delta$ would each contain a basepoint, see Figure~\ref{fig:14}. Hence the $\delta$-disk essentially intersects the knot once(in the solid torus $L-N(\mathcal{B})$). Hence the $\delta$ curve is essential on $\partial N(\mathcal{B})$ and bounds a disk in $L-N(\mathcal{B})$, thus representing the knot meridian.
    ~\\
    
    \begin{figure}[!hbt]
        \begin{overpic}{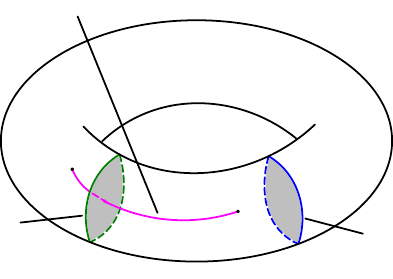}
            \put(1,13){$\color{green} \delta$}
            \put(93,10){$\color{ao} \beta$}
            \put(1,70){(1,1)-knot arc inside the solid torus}
            \put(14.5,28){$z$}
            \put(57,19){$w$}
            \put(100,25){$\Sigma$}
        \end{overpic}
        \caption{The $\delta$-disk and the (1,1)-knot}
        \label{fig:14}
    \end{figure}
    
    Denote the two horizontal boundary components of $N(\mathcal{B})$ as $\partial_h^{+\prime},\partial_h^{-\prime}$, so that the $\delta$-annulus is attached to $\partial_h^{+\prime}$. Since $\delta$ also represents the knot meridian, $\partial_h^{+\prime}$ is divided by $\delta$ into two annuli. See the left two figures in Figure~\ref{fig:15}, where the top and bottom sink bigon holes are identified to form the boundary torus of $N(\mathcal{B})$. Now suppose $\mathcal{B}'$ fully carries a lamination. According to the branch direction of $\delta$ (see Figure~\ref{fig:15}), we can do different operations to eliminate the $\delta$-annulus and get a lamination fully carried by $\mathcal{B}$. 
    
    \begin{figure}[!thb]
        \begin{overpic}[scale=0.9]{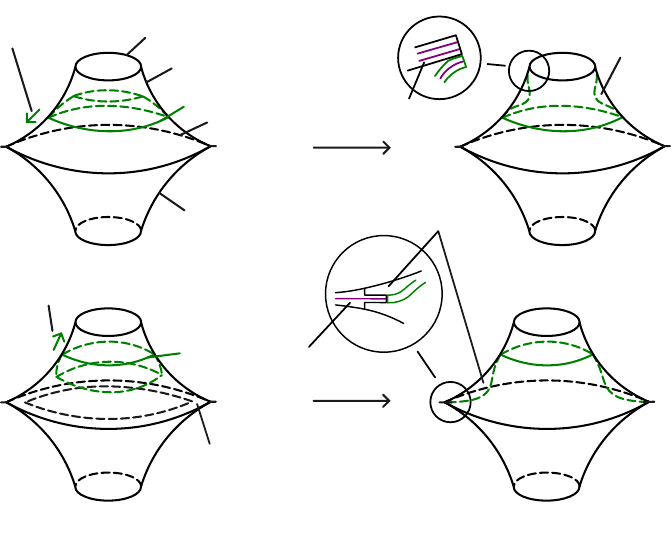}
            \put(20,75){sink bigon hole}
            \put(26,69){$\partial_h^{+\prime}$}
            \put(32,60){cusp}
            \put(28,46){$\partial_h^{-\prime}$}
            \put(28,63){$\color{green}\delta$}
            \put(27.2,26){$\color{green}\delta$}
            \put(-3,73){br. direction}
            \put(0,35){branch direction}
            \put(86,72){annulus bubble}
            \put(53,46){annulus bubble}
            \put(28,10){annulus attached to cusp}
            \put(14,0){$\mathcal{B''}$}
            \put(80,0){$\mathcal{B'''}$}
            \put(53,61){lamination}
            \put(43,24){$F'$}
        \end{overpic}
        \caption{Eliminating annulus}
        \label{fig:15}
    \end{figure}

    If the branch direction points to the cusp (Figure~\ref{fig:15} top), we can attach a neighborhood of the boundary circle on the $\delta$-annulus to a neighborhood of the boundary circle in the sink bigon. The new branched surface (as depicted in Figure~\ref{fig:15} top-right) then fully carries a lamination while containing an annulus bubble. By collapsing the annulus bubble, the lamination collapses to become a lamination $\mathcal{N}$ fully carried by $\mathcal{B}$.
    
    Now suppose the branch direction points to the boundary circle in sink bigon (Figure~\ref{fig:15} bottom). We show that we can still modify the lamination $\mathcal{N'}$ fully carried by $\mathcal{B'}$. Recall after some isotopy we can assume that the metric completion of $N(\mathcal{B'})-\mathcal{N'}$ consists of $I$-bundles. Consider the $I$-bundle whose boundary contains the vertical boundary corresponding to the cusp of $N(\mathcal{B})$. Since our branched surface is transversely oriented, this $I$-bundle is some $F\times [0,1]$, where $F$ is oriented and has a circle boundary $C$ corresponding to the cusp ($F$ itself may have multiple boundary components). We can extend $F\times \frac{1}{2}$ along $C$ (i.e. attaching a small annulus to $C$) and add this leaf $F'$ to $\mathcal{N'}$. The resulting lamination $\mathcal{N''} = \mathcal{N'}\cup \{F'\}$ is fully carried by some branched surface $\mathcal{B''}$ which is obtained by attaching an annulus to the cusp of $\mathcal{B'}$ (see Figure~\ref{fig:15} bottom-left). Since attaching (thickened) annuli along boundary does not change the topology of $N(\mathcal{B})$, we have $N(\mathcal{B''})$ also homeomorphic to the knot complement. Moreover, all three boundary circles of $\mathcal{B''}$ (one each coming from the sink disk, the $\delta$-annulus, and the cusp) correspond to the knot meridian. Hence $\mathcal{B''}$ cannot carry disks bounded by these boundary components. Now by the gluing lemma Lemma~\ref{lem:4} we can glue the two boundary circles corresponding to $\delta$ and cusp, and get a branched surface $\mathcal{B'''}$ fully carrying a lamination (see Figure~\ref{fig:15} bottom-right). One can then collapse the annulus bubble of $\mathcal{B'''}$ to get $\mathcal{B}$, while the lamination is collapsed to some $\mathcal{N}$ fully carried by $\mathcal{B}$.
\end{proof}

\textbf{Step 2: Taking the difference.} Let $\mathcal{C}$ be the branched surface associated to $(\Sigma,\alpha,\delta,z,w)$. We can remove a small open disk in the interior of the $\delta$ disk to get a branched surface $\mathcal{C'}$ (which still fully carries a lamination if $\mathcal{C}$ does). We can then think of the branched surface $\mathcal{C'}$ as a sub-branched surface of $\mathcal{B'}$ (i.e. there exists an embedding $\mathcal{C'}\hookrightarrow \mathcal{B'}$ such that each sector of $\mathcal{C'}$ is a union of sectors of $\mathcal{B'}$). We consider the complement $\mathcal{A}=\mathcal{B'}-\mathcal{C'}$. $\mathcal{A}$ consists of the $\beta$-disk and sectors on the torus corresponding to the $(\alpha,\delta)$-source bigon minus the $(\alpha,\beta)$-source bigon (see the shaded regions in Figure~\ref{fig:12}; notice the two bigons are always nested in this way, since when we place $\alpha$ in standard position, the (innermost) rainbow $\delta$-arc exists (by non-simpleness supposition of (1,1)-diagram) and is amid the rainbow $\beta$-arcs (lying in a $\beta$-parallel sector)).

\begin{figure}[!hbt]
    \begin{overpic}[scale=0.6]{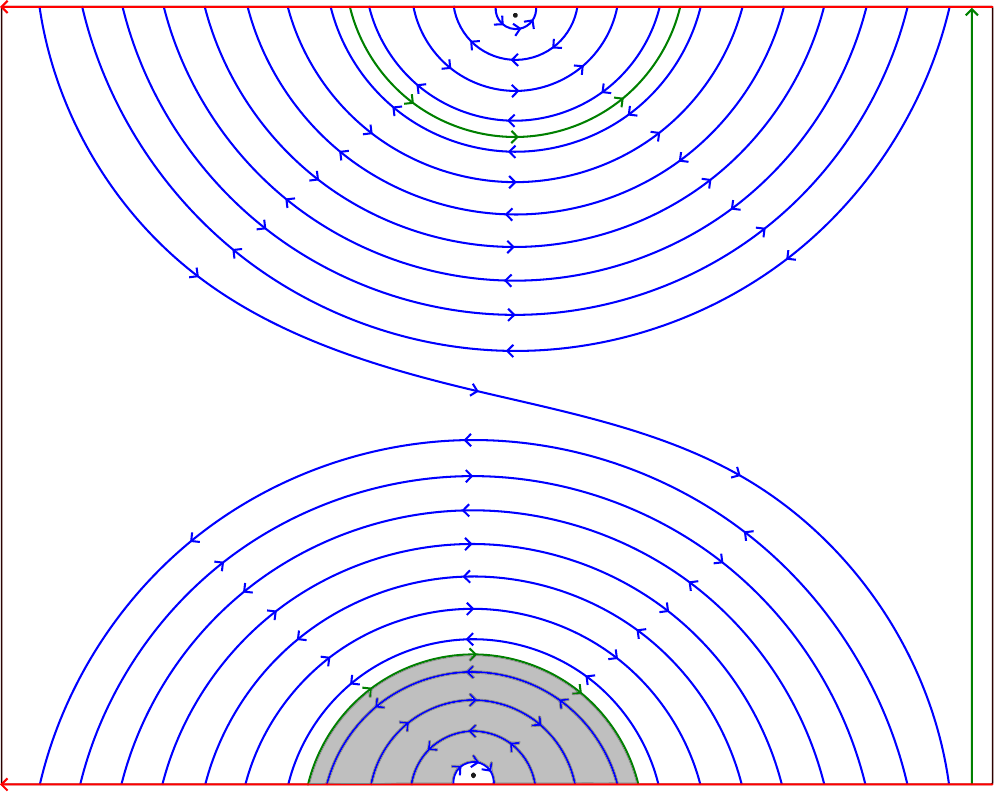}
        \put(95,60){$\color{green} \delta$}
        \put(80,30){$\color{ao} \beta$}
        \put(98,-1.8){$\color{red} \alpha$}
        \put(29.5,-1.7){\tiny$M$}
        \put(63,-1.7){\tiny$N$}
        \put(34,80){\tiny$N$}
        \put(97,80){\tiny$M$}
    \end{overpic}
    \caption{Knot (23,11,1,7)}
    \label{fig:12}
\end{figure}

$\mathcal{A}$ alone cannot pass on laminations, so we need some further constructions. We construct a branched surface $\mathcal{A'}$ consisting of $\mathcal{A}$ and the horizontal boundary $\partial_h N(\mathcal{C'})$, where $\mathcal{A}$ intersects the two horizontal boundary components in the same branch direction as if intersecting with $\mathcal{C'}$, see Figure~\ref{fig:11}. Denote the two horizontal boundary components $\partial_h^+$ and $\partial_h^-$. These are pairs of pants each with three boundary circles. We denote the boundary circles coming from the cusp of $\mathcal{C'}$ as $l_{\mathcal{C'}}^{\pm}$ (the $\pm$ sign here indicates if it's the boundary circle of $\partial_h^+$ or $\partial_h^-$), the boundary circles in the sink bigon as $l_{s}^{\pm}$, and the boundary circles at the $\delta$-annulus as $l_{\delta}^{\pm}$. Suppose the $\beta$-disk (of $\mathcal{A}$) is to be attached to $\partial_h^+$. 

\begin{figure}[!hbt]
    \begin{overpic}[scale=0.6]{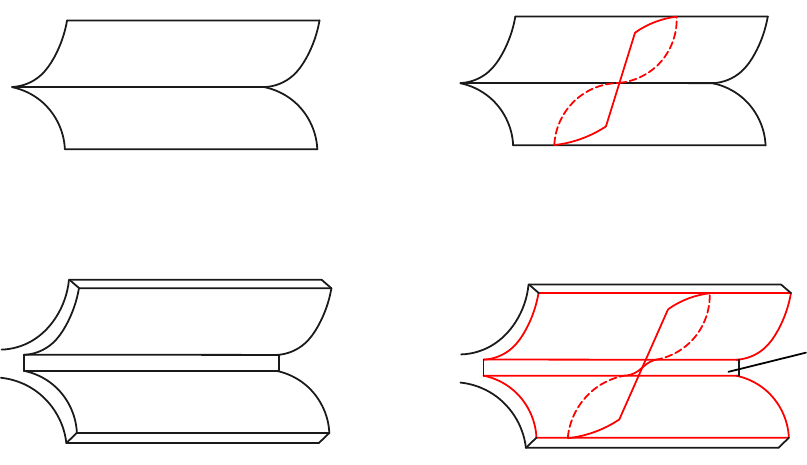}
        \put(21,31){$\mathcal{C'}$}
        \put(17,-4){$N(\mathcal{C'})$}
        \put(63,31){$\mathcal{B}$ (with $\mathcal{A}$ in red)}
        \put(69,-4){$\mathcal{A'}$ (in red)}
        \put(100,12){\small boundary train track}
    \end{overpic}
    \caption{Attaching $\mathcal{A}$ to $\partial_h N(\mathcal{C'})$}
    \label{fig:11}
\end{figure}

We remark that attaching $\mathcal{A}$ to $\partial_h^{\pm}$ gives us directions for how to extend laminations carried by $\mathcal{A}$ in $N(\mathcal{B})$, eventually allowing us to prove Lemma~\ref{lem:12} below. On the other hand, we use the horizontal boundary of $N(\mathcal{C'})$ instead of $\mathcal{C'}$ itself, since in this way we could make the branch locus simpler, so that we could better control the double points, which is crucial in our proof. More details will be explained in subsection~\ref{subsec:3.5}, especially in the remark following Figure~\ref{fig:21}.

\begin{lem}
    Suppose $\mathcal{A'}$ and $\mathcal{C'}$ fully carry laminations, then $\mathcal{B'}$ also fully carries a lamination.
    \label{lem:12}
\end{lem}

\begin{proof}[Proof of Lemma~\ref{lem:12}]
    We can glue ${\partial}_h^{\pm}$ in $\mathcal{A'}$ back to the corresponding horizontal boundary components of $N(\mathcal{C'})$. Now the lamination $\mathcal{M}$ fully carried by $\mathcal{A'}$ can be regarded as a closed subset of $N(\mathcal{B'})$, containing surface leaves transverse to the $I$-fibers. However, $\mathcal{M}$ is not yet a lamination of $N(\mathcal{B'})$ since its leaves are not properly embedded at the $\mathcal{C'}$-cusp (here $\mathcal{A'}$ has a boundary train track $\mathcal{T}$ as in Figure~\ref{fig:13} left, see also Figure~\ref{fig:11} bottom-right). The lamination $\mathcal{M}$ intersects the vertical boundary annulus of $N(\mathcal{C'})$ at the cusp in a 1-dimensional lamination. 
    
    Now by our supposition $\mathcal{C'}$ fully carries a lamination $\mathcal{L}$. Similarly to before, after some isotopy we can assume that the metric completion of $N(\mathcal{C'})-\mathcal{L}$ consists of $I$-bundles, and that the $I$-bundle whose boundary contains the vertical annulus at $\mathcal{C'}$-cusp is some $G\times I$ for $G$ an oriented surface. Moreover, since $\mathcal{C}$ does not contain a disk of contact (Lemma~\ref{lem:6}), neither does $\mathcal{C'}$ (which is obtained by removing a disk from $\mathcal{C}$). Hence $G$ cannot be a disk. It is then a standard argument that $G\times I$ can be given a lamination transverse to the $I$-fibers, such that when restricted to the boundary vertical annulus at the $\mathcal{C'}$-cusp it is any given 1-dimensional lamination. (If $G$ has positive genus, this is~\cite{li2002laminar}, Lemma 3.1 and 3.2; if $G$ is planar, it must have multiple boundary components or ends, and we can trivially extend the 1-dimensional lamination by sending it to one of the other boundary components or ends.) By setting the 1-dimensional lamination to be the same as that of $\mathcal{M}$ at $\mathcal{C'}$-cusp, we can paste our lamination in $G\times I$ to $\mathcal{M}$ along the vertical annulus. We then get a lamination $\mathcal{M'}$ carried by $\mathcal{B'}$. Now $\mathcal{N'}=\mathcal{L}\cup\mathcal{M'}$ is a lamination fully carried by $\mathcal{B'}$.
\end{proof}

\textbf{Step 3: Collapsing the boundary train track.} For technical reasons (to apply Lemma~\ref{lem:3}) we need to ``collapse'' the boundary train track of $\mathcal{A'}$ to circles. Recall that each of $\partial^{\pm}_h$ has 3 boundary circles. When attaching $\mathcal{A}$ to them to obtain $\mathcal{A'}$, $l_{\mathcal{C'}}^{\pm}$ are the only boundary circles that $\mathcal{A}$ intersects. Hence besides the circle boundary components, $\mathcal{A'}$ has a boundary train track at the $\mathcal{C'}$-cusp (one can spot this cusp on the (1,1)-diagram as $\alpha\cup\delta-\partial((\alpha,\delta)$-source bigon)), consisting of $l_{\mathcal{C'}}^{\pm}$ and two short boundary arcs from $\mathcal{A}$, see Figure~\ref{fig:11} and Figure~\ref{fig:13} left. Denote the boundary train track as $\mathcal{T}$. 

\begin{figure}[!hbt]
    \begin{overpic}[scale=0.6]{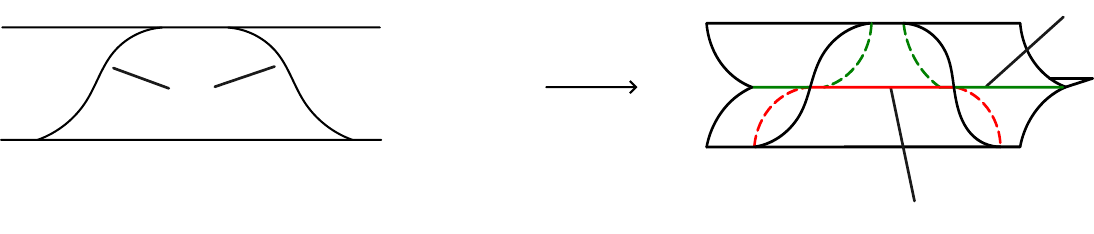}
        \put(40,12){$\times \;I$}
        \put(35.5,7){$l_{\mathcal{C'}}^+$}
        \put(35.5,18){$l_{\mathcal{C'}}^-$}
        \put(16,11.5){$\mathcal{A}$}
        \put(83,1){$\color{red}\alpha$}
        \put(97.5,20){$\color{green}\delta$}
        \put(72,14){\tiny$M$}
        \put(87,14){\tiny$N$}
        \put(6,3){\small boundary train track}
    \end{overpic}
\caption{Collapsing at the $\mathcal{C'}$-cusp}
\label{fig:13}
\end{figure}

Now we collapse/pinch $\mathcal{T}$ into a single circle as shown in Figure~\ref{fig:13}. To give a detailed description, take a neighborhood of $\mathcal{T}$ in $\mathcal{A'}$ and denote it as $\mathcal{T}\times I$ (as shown in Figure~\ref{fig:13} left), with $\mathcal{T}\times\{0\}$ representing the actual boundary of $\mathcal{A'}$. Our ``collapse'' is then described as a quotient operation on $\mathcal{T}\times [0,1/2]$, where we keep the $[0,1/2]$ factor unchanged and quotient $\mathcal{T}$. More precisely, we map the two $\mathcal{A}$-arcs of the boundary train track to two points, and identify the corresponding $l^{\pm}_{\mathcal{C'}}$-arcs cut out by the $\mathcal{A}$-arcs. See Figure~\ref{fig:13}, where we note that after collapsing the two double points $M,N$ (obtained by collapsing the $\mathcal{A}$-arcs) actually correspond to the vertices of the $(\alpha,\delta)$-source bigon (see them also in Figure~\ref{fig:12}). 

The resulting branched surface $\mathcal{A''}$ has a circle boundary at the $\mathcal{C'}$-cusp instead of a train track, and has branch locus 2 cusp circles (one is $\delta$ and the other is ($\alpha\cup\beta-\partial((\alpha,\beta)$-source bigon)). See Figure~\ref{fig:13} right, where the $\alpha$-locus is joined with $\beta$ at the $(\alpha,\beta)$-source bigon away from the $\mathcal{C'}$-cusp).

\begin{lem}
    $\mathcal{A'}$ fully carries a lamination if $\mathcal{A''}$ does.
    \label{lem:13}
\end{lem}

\begin{proof}[Proof of Lemma~\ref{lem:13}]
    We show that we can construct an embedding $\mathcal{A'}\hookrightarrow \mathcal{A''}$. Recall in the construction of $\mathcal{A''}$ we only changed $\mathcal{T}\times [0,1/2]\subset \mathcal{A'}$, a small neighborhood of $\mathcal{T}$. The embedding can then be defined to be the identity outside $\mathcal{T}\times [0,1]$, while mapping $\mathcal{T}\times [0,1]$ to $\mathcal{T}\times [3/4,1]$ via a linear map of the second factor. The lemma follows by restricting the lamination fully carried by $\mathcal{A''}$ to $\mathcal{A'}$.
\end{proof}

\begin{lem}
    $\mathcal{A''}$ fully carries a lamination.
    \label{lem:14}
\end{lem}

We defer the proof of Lemma~\ref{lem:14} to subsection~\ref{subsec:3.5} where we heavily use splittings. On the other hand, by far we already have all the ingredients to prove Proposition~\ref{prop:3}, and subsequently Proposition~\ref{prop:1}.

\begin{proof}[Proof of Proposition~\ref{prop:3}]
    By Lemma~\ref{lem:14} and Lemma~\ref{lem:13} we know $\mathcal{A'}$ fully carries a lamination. Since $\mathcal{C}$ fully carries a lamination by our hypothesis, $\mathcal{C'}$, obtained by removing a disk in the interior of a sector of $\mathcal{C}$, also fully carries a lamination. Now by Lemma~\ref{lem:12} and Lemma~\ref{lem:11} $\mathcal{B}$ fully carries a lamination, as required. 
\end{proof}

\begin{proof}[Proof of Proposition~\ref{prop:1}]
    Let $(\Sigma,\alpha,\beta,z,w)$ be a non-simple reduced (1,1)-diagram. If it is primitive, then by Proposition~\ref{prop:2} the associated branched surface fully carries a lamination. If it is not, then by Proposition~\ref{prop:hierarchy} there is a finite hierarchy reducing it to a primitive one, whose associated branched surface fully carries a lamination. Now by Proposition~\ref{prop:3} and induction every reduced (1,1)-diagram in this hierarchy has its associated branched surface fully carrying laminations.
\end{proof}

\subsection{Splittings and sink tube push}
\label{subsec:3.4}

In this subsection we prove Proposition~\ref{prop:2}. In general we are to modify the branched surfaces to be sink disk free by splittings. We first recall some basic definitions.

\begin{figure}[!hbt]
    \begin{overpic}{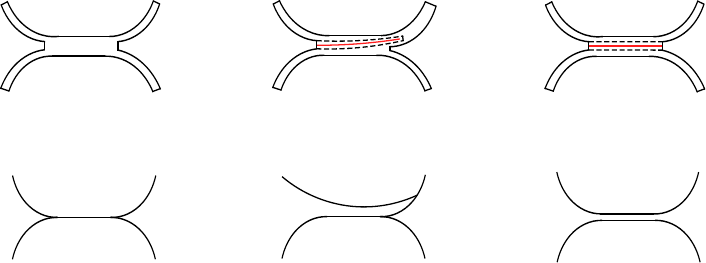}
        
    \end{overpic}
    \caption{Splittings}
    \label{fig:16}
\end{figure}

A \textbf{splitting} of a branched surface $\mathcal{B}$ along an oriented compact surface $S\subset N(\mathcal{B})$ (where $S$ is transverse to the $I$-fibers) results in a branched surface $\mathcal{B}^S$, such that there is an inclusion of $I$-bundle $N(\mathcal{B}^S)\subset N(\mathcal{B})$ respecting the $I$-fibers, where $N(\mathcal{B})\backslash N(\mathcal{B}^S)=N(S)\cong S\times I$. Notice that if $\partial S\cap \partial_v N(\mathcal{B})=\varnothing$, then the splitting is just creating a bubble, so we generally hope this does not happen. Typically, $S$ is a disk and $\partial S$ intersects $\partial_v N(\mathcal{B})$ in arcs, and the splitting is locally modelled on one of the pictures in Figure~\ref{fig:16}.

When the branched surface is obtained by pasting sectors (mostly disks in this paper) to a fixed surface $\Sigma$, we can describe the branch locus as curves on $\Sigma$. Then there is a special kind of splitting that gains importance in practice, which we call ``\textbf{pushing arcs}''. These splittings are along disks and are locally modelled on the middle picture of Figure~\ref{fig:16}. When observed on $\Sigma$, it looks like one arc (say, from curve $\alpha$) is pushed over another arc (from $\gamma$) onto the sector bounded by $\gamma$, and we say that the $\alpha$-arc is \textit{pushed onto} the $\gamma$-arc. See Figure~\ref{fig:17} for a 3-dimensional picture. We remark that one can also push arcs from the same curve.

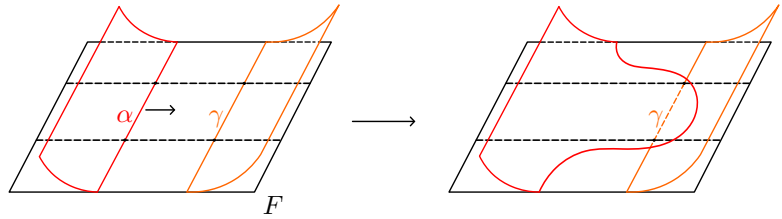
\begin{figure}[!hbt]
    \begin{overpic}[scale=0.7]{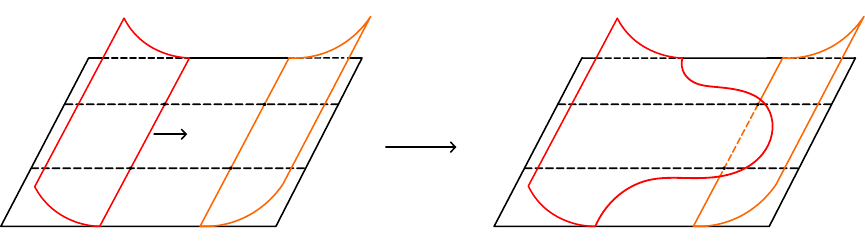}
        \put(14,11){$\color{red} \alpha$}
        \put(26,11){$\color{or} \gamma$}
        \put(83,11){$\color{or} \gamma$}
        \put(33,-1){$F$}
    \end{overpic}
    \caption{Pushing the $\alpha$-arc onto the $\gamma$-arc}
    \label{fig:17}
\end{figure}

The major advantage of talking about ``pushing arcs'' is that we can model these 3-dimensional operations on some 2-dimensional figure $\Sigma$, while the model fully reflects the operations on the branched surface. In the following discussions unless specified we will talk about the operations on the 2-dimensional model rather than directly describing them on the branched surfaces.

Now we introduce what we call the ``sink tube push''. Let $(\Sigma,\alpha,\beta,z,w)$ be a non-simple reduced (1,1)-diagram, and $\mathcal{B}$ be its associated branched surface. Recall from Definition~\ref{def:2} and Proposition~\ref{lem:2} that the $\beta$-sink tube is the connected region containing all the $\beta$-sink sectors. This tube is bounded by two $\beta$-sink $\alpha$-arcs (one of them being the $\alpha$-arc of the sink bigon) and two disjoint long $\beta$-arcs. Now we consider the location of the source bigon (since it is removed, we generally cannot push the boundary arcs of it). The boundary $\beta$-arc of the \textit{source} bigon belongs to at most one of the two long $\beta$-arcs of the sink tube. Supposing one long $\beta$-arc does contain the $\beta$-arc of the source bigon, then our ``$\beta$\textbf{-sink tube push}'' is to push the other $\beta$-arc \textit{onto} this arc. See Figure~\ref{fig:18} below for illustrations of this operation, where in $(b)$ the dashed arcs are pushed \textit{onto} the solid arcs (and are no longer on the surface $\Sigma$).

\begin{figure}[!hbt]
    \begin{overpic}[scale=0.75]{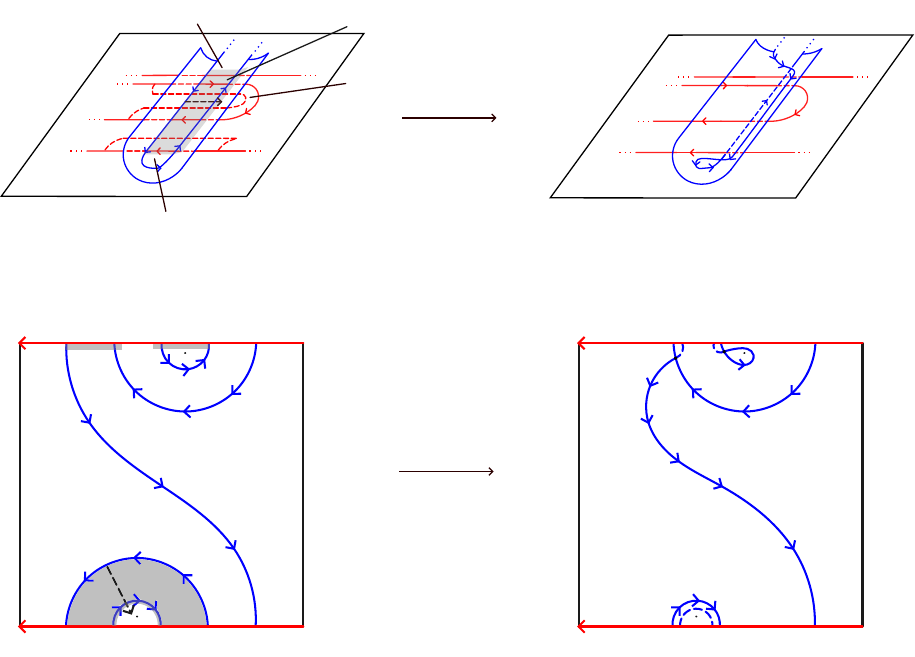}
        \put(0,70.7){polygon with sink $\alpha$-arc}
        \put(38,70){spliting disk $S$}
        \put(38.8,62){source bigon}
        \put(18,46.5){sink bigon}
        \put(20,52){$\color{ao}\beta$}
        \put(26,53.5){$\color{red}\alpha$}
        \put(22,40){$(a)$ $\beta$-sink tube push in branched surface}
        \put(22,-2){$(b)$ $\beta$-sink tube push in (1,1)-diagram (of $4_1$)}
        \put(77.4,32){\Tiny $X_1$}
        \put(70.5,32.6){\Tiny $X_2$}
    \end{overpic}
    \caption{Sink tube push}
    \label{fig:18}
\end{figure}

~\\
\begin{rems}~\
    \begin{enumerate}[(1)]
        \item Our $\beta$-sink tube push is only defined when one of the long $\beta$-arcs does contain the $\beta$-arc of the source bigon. We will always reduce to this situation when we need a sink tube push.
        \item To define a $\beta$-sink tube push, we only need a branched surface where a (disk) sector is attached to a surface $\Sigma$ along some (essential) curve $\beta \subset \Sigma$, and a well-defined ``$\beta$-sink tube'' from the information on $\Sigma$. Later we'll define similar operations on branched surfaces modified from $\mathcal{B}$, and we may still call them ``sink tube pushes''.
    \end{enumerate}
    \label{rem:4}
\end{rems}

To check if a branched surface is sink disk free, we generally need to examine all its sectors. However, when there's no ``trivial'' sink disk (we'll explain this shortly), we can instead examine the sink disk candidates around the double points. It turns out that there's only one candidate around each double point:

\begin{defn}
    Let $\mathcal{B}$ be a branched surface and $X$ a double point in its branch locus. We say a branch sector $D$ of $\mathcal{B}$ is \textit{around $X$} if $X\in \partial D$. 
    
    We can spot all sectors around $X$ in a local picture of $X$, see Figure~\ref{fig:6}.$(a)$. By enumerating these sectors, we can see that there is a unique sector $D$, s.t. there exists a neighborhood $N(X)$ of $X$, where the branch locus arc $N(X)\cap \partial D$ always has its branch direction pointing into $D$. See the upper-left corner of Figure~\ref{fig:6}.$(a)$. We call this sector $D$ the \textbf{sink corner} of the double point $X$.
    \label{def:10}
\end{defn}

\begin{lem}
    Let $\mathcal{B}$ be a branched surface. If the branch locus of $\mathcal{B}$ does not contain a simple circle disjoint from others, then a sink disk of $\mathcal{B}$ must be the sink corner of some double point.
    \label{lem:15}
\end{lem}

\begin{proof}
    Suppose $D$ is a sink disk of $\mathcal{B}$. If there's no double point in $\partial D$, then $\partial D$ is a branch locus circle that's disjoint from other branch locus components, contradicting our supposition. Hence there's some double point $X\in \partial D$, and $D$ is then the sink corner of $X$.
\end{proof}

Now we're in position to prove Proposition~\ref{prop:2}.

\begin{proof}[Proof of Proposition 3.4]
    Suppose $(\Sigma,\alpha,\beta,z,w)$ is primitive and $\alpha$ is placed in standard position. We first claim that, if we use the 4-tuple notation $(p,q,r,s)$ as introduced in~\cite{rasmussen2005knot} (see Figure~\ref{fig:1}) to identify the primitive (1,1)-diagram, then $s$ is uniquely determined by $(p,q,r)$. We can look along $\alpha$ at the signs of consecutive intersections of $\alpha$ and $\beta$. At the rainbow arcs the signs are alternating; while at the vertical arcs the signs are the same. It follows that there's only one maximal group of consecutive intersections of the same sign that contains more than one element (that is, the one containing all vertical arc intersections, see Figure~\ref{fig:19} for an example, where the maximal group size is 2). Since we can observe this maximal group on both the roof and the bottom of the square, there's a unique legible way of gluing back the roof and the bottom, so that the maximal groups on the roof and the bottom are identified. This indicates that the gluing twist $s$ is determined by $(p,q,r)$, the parameters of the square. Moreover, in a primitive (1,1)-diagram the sink bigon and the source bigon share an endpoint, since they're at the middle of the (maximal) alternating intersections, see Figure~\ref{fig:19} and Figure~\ref{fig:20}.

    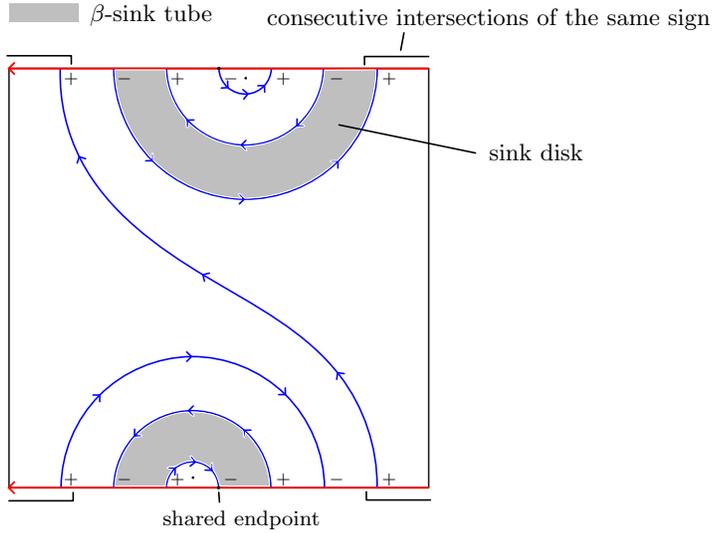
\begin{figure}[!hbt]
        \begin{overpic}[scale=0.7]{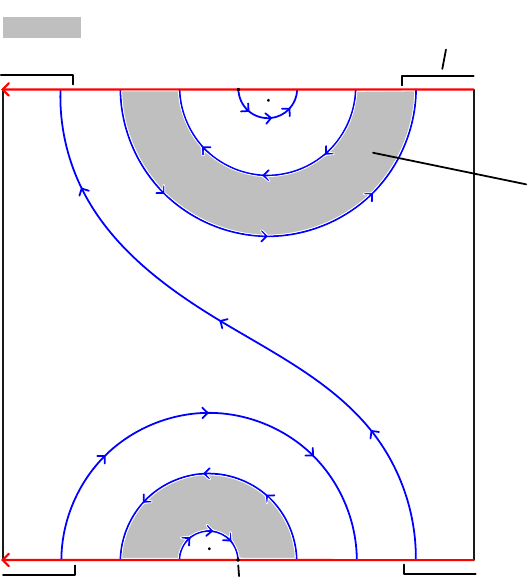}
            \put(16,93.7){\small $\beta$-sink tube}
            \put(50,93){\small consecutive intersections of the same sign}
            \put(92.6,67){\small sink disk}
            \put(30,-3){\Small shared endpoint}
            \put(11,82){\tiny $+$}
            \put(21.2,82){\tiny $-$}
            \put(31.4,82){\tiny $+$}
            \put(41.6,82){\tiny $-$}
            \put(51.7,82){\tiny $+$}
            \put(61.9,82){\tiny $-$}
            \put(72,82){\tiny $+$}
            \put(11,5){\tiny $+$}
            \put(21.2,5){\tiny $-$}
            \put(31.4,5){\tiny $+$}
            \put(41.6,5){\tiny $-$}
            \put(51.7,5){\tiny $+$}
            \put(61.9,5){\tiny $-$}
            \put(72,5){\tiny $+$}
        \end{overpic}
        \caption{A primitive (1,1)-diagram for knot $5_2$}
        \label{fig:19}
    \end{figure}

    There might be sink disks in the branched surface $\mathcal{B}$ associated to $(\Sigma,\alpha,\beta,z,w)$. Recall that by our construction branch direction is always pointing to the left of the oriented curve, so a disk sector on $\Sigma$ is a sink disk if and only if its boundary goes in the anti-clockwise direction, see Figure~\ref{fig:19} for an example. In other words, it is a sink disk if and only if all its boundary arcs are ($\alpha$- or $\beta$-)sink. Since there is a hole in the sink bigon, all sink disks on $\Sigma$ are sectors of the ($\beta$-)sink tube.

    \begin{figure}[!hbt]
        \begin{overpic}[scale=0.7]{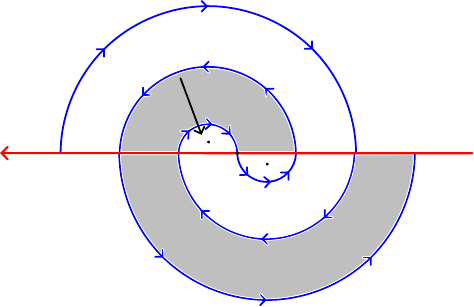}
            
        \end{overpic}
        \caption{Sink tube spiral}
        \label{fig:20}
    \end{figure}

    If there are sink disks, we can then do a ``sink tube push'' to eliminate them. Since the sink and source bigons share an endpoint, by a short induction on the number of rainbow arcs we know the sink tube is a spiral whose long $\beta$-arcs consist of rainbow arcs, see Figure~\ref{fig:20}. One then observes that all the sink sectors would have their ``outer'' rainbow $\beta$-arcs pushed. Hence after the operation all remaining rainbow arcs on $\Sigma$ are of the same direction.

    Now we can spot the double points of the branched surface $\tilde{\mathcal{B}}$ after our operation and check their sink corners. The remaining intersections of $\alpha$ and $\beta$ on $\Sigma$ still serve as double points; besides, the sink tube push would create two new double points, $X_1$ in the sink bigon and $X_2$ in the polygons. The sink corner of $X_1$ is an annulus because of the hole in the sink bigon, thus not a sink disk. The sink corner $D_2$ of $X_2$ is a disk and actually comes from a polygon on the (1,1)-diagram. Now since our sink tube push does not move an entire vertical arc (we only moved entire arcs in the spiral), $D_2$ still has a boundary $\beta$-arc that comes from a vertical arc on the (1,1)-diagram. Moreover, there are multiple such boundary $\beta$-arcs since one needs to go back to the same side of $\alpha$ (i.e. roof or bottom) while traveling along $\partial D_2$. See Figure~\ref{fig:18}.$(b)$ for an example, where $D_2$ is actually bounded by the same ``vertical'' arc twice. Since our (1,1)-diagram is primitive, the vertical arcs are in the same direction, thus $D_2$ also cannot be a sink disk. The sink corners of remaining $(\alpha,\beta)$ intersections on $\Sigma$ are all on the torus. If the sink corner has boundary $\beta$-arcs coming from vertical arcs, then as before there are multiple such boundary arcs and they are in the same direction; hence such a sink corner cannot be a sink disk. If the sink corner does not have boundary $\beta$-arcs coming from vertical arcs, then since it is not a bigon, there are multiple ``rainbow'' boundary $\beta$-arcs, and they are also in the same direction; such a sink corner also cannot be a sink disk.

    The branch locus of $\mathcal{B}$ is a single immersed curve obtained by joining $\alpha$ and $\beta$ at endpoints of the source bigon. It then follows that the branch locus of $\tilde{\mathcal{B}}$ is still a single immersed curve. Now that we have examined the sink corners of all double points of $\tilde{\mathcal{B}}$, we can use Lemma~\ref{lem:15} to conclude that $\tilde{\mathcal{B}}$ is sink disk free.
    
    Since our splitting is locally modelled on Figure~\ref{fig:16} middle, it does not change the topology of horizontal boundary, so $\tilde{\mathcal{B}}$ still has horizontal boundary components two annuli. Since $\mathcal{B}$ is transversely oriented with $N(\mathcal{B})$ homeomorphic to the (1,1)-knot complement, any closed surface carried by $\mathcal{B}$ is orientable and non-separating in the total space $L$ (a transverse closed curve would connect the two sides). Since $H_1(L,\mathbb{Q})=0$, there is no such closed surface. In particular $\mathcal{B}$, thus $\tilde{\mathcal{B}}$, cannot carry any torus. Now by Lemma~\ref{lem:3} our $\tilde{\mathcal{B}}$ fully carries a lamination, thus $\mathcal{B}$ also fully carries a lamination.
\end{proof}

\subsection{More splittings, and rational tangle strands}
\label{subsec:3.5}

It remains for us to prove Lemma~\ref{lem:14}. Our strategy is still to push arcs to eliminate sink disks, and then apply Lemma~\ref{lem:3}. In fact, the reason we modify the branched surface to $\mathcal{A''}$ is that the sink disks (double points) are better controlled.

\begin{proof}[Proof of Lemma~\ref{lem:14}]
    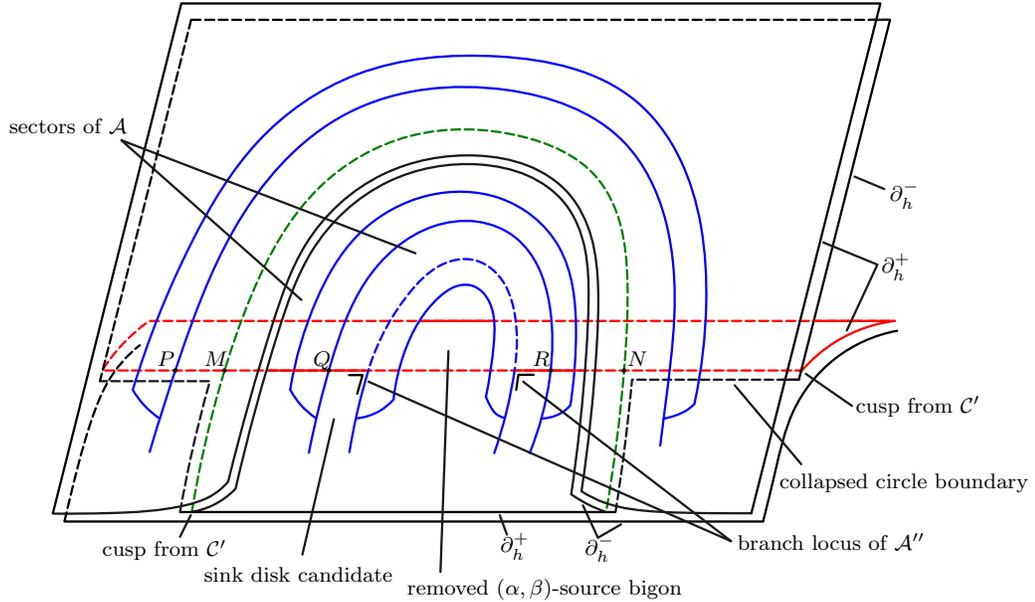
\begin{figure}[!hbt]
        \begin{overpic}{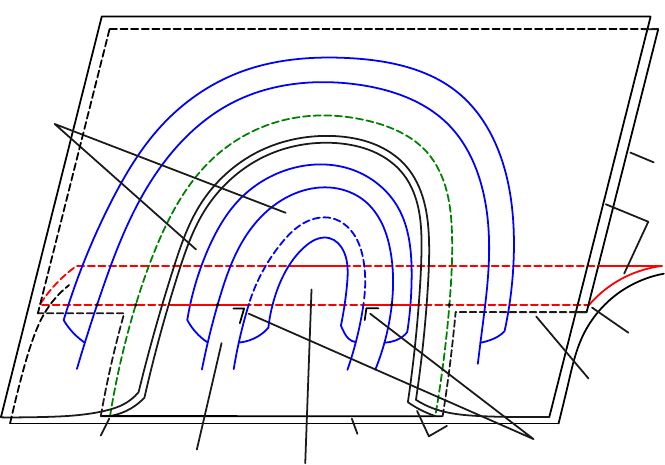}
            \put(18,-1){\Small sink disk candidate}
            \put(42,-2.5){\Small removed $(\alpha,\beta)$-source bigon}
            \put(6,2){\Small cusp from $\mathcal{C'}$}
            \put(53,2.5){\Small $\partial_h^+$}
            \put(63,2.2){\Small $\partial_h^-$}
            \put(95,19){\Small cusp from $\mathcal{C'}$}
            \put(98,35.5){\Small $\partial_h^+$}
            \put(99,44){\Small $\partial_h^-$}
            \put(81,2.8){\Small branch locus of $\mathcal{A''}$}
            \put(18,24.5){\Small $M$}
            \put(68,24.5){\Small $N$}
            \put(12.5,24.5){\Small $P$}
            \put(30.9,24.5){\Small $Q$}
            \put(56.8,24.5){\Small $R$}
            \put(-5,52){\Small sectors of $\mathcal{A}$}
            \put(86,10){\Small collapsed circle boundary}
        \end{overpic}
        \caption{$\mathcal{A''}$ near the $(\alpha,\delta)$-source bigon}
        \label{fig:21}
    \end{figure}

    We first find the sink disks in $\mathcal{A''}$. Recall from the construction that the branch locus of $\mathcal{A''}$ consists of two cusp circles, one corresponding to $\delta$ and the other ($\alpha\cup\beta-(\alpha,\beta)$-source bigon). Since the two circles do intersect at $M$ and $N$, the vertices of the $(\alpha,\delta)$-source bigon (see Figure~\ref{fig:13}), by Lemma~\ref{lem:15} we can just check the sink corners of the double points of $\mathcal{A''}$. The two double points $M,N$ at the $\mathcal{C'}$-cusp share the same sink corner, which is the annulus obtained by pinching the boundary train track of $\mathcal{A'}$, not a sink disk. The rest of the $\delta$-locus away from the $\mathcal{C'}$-cusp is the boundary of the $(\alpha,\delta)$-source bigon, where $\mathcal{A}$ is attached to $\partial_h^-$. Notice this is the only part of $\mathcal{A}$ that is attached to $\partial_h^-$, so there is no other double point on $\delta$. Remaining branch locus not discussed contains roughly the boundary $\alpha$-arc of the $(\alpha,\delta)$-source bigon and the $\beta$-curve, minus the boundary of the $((\alpha,\beta)$-source bigon), at the vertices of which the two curves are joined (see Figure~\ref{fig:21}). If we place $\alpha$ in standard position, then we can spot the double points, as the $\beta$ curve intersects the $(\alpha,\delta)$-source bigon in alternating rainbow arcs, see the shaded region in Figure~\ref{fig:12} or $Q,R$ in Figure~\ref{fig:21}. One can see that the sink corners of double points like $Q,R$ are candidates for sink disks of $\mathcal{A''}$, as in Figure~\ref{fig:21}.~\\

    \textbf{Remark}
    \begin{enumerate}
        \item The $\beta$-rainbow arcs inside the $(\alpha,\delta)$-source bigon must be alternating, since $\delta$ passes through all parallel sectors.
        \item In Figure~\ref{fig:21}, $Q,R$ are double points of $\mathcal{A''}$. However, $P$ is not a double point of $\mathcal{A''}$. In fact, here the red $\alpha$-locus is at the collapsed $\mathcal{C'}$-cusp, while the blue $\beta$-disk is attached to $\partial_h^+$ away from the cusp, so the two loci don't actually intersect. \textbf{This is the very point for us to do all these inductive constructions.} By taking horizontal boundary $\partial_h^{\pm}$ and attaching $\mathcal{A}$ to them, we manage to separate the branch locus outside the $(\alpha,\delta)$-source bigon. As a result, we only need to take care of the double points inside the $(\alpha,\delta)$-source bigon, which is a lot easier.
    \end{enumerate}
    ~\\

    Now we begin to push arcs to eliminate possible sink disks. Again, we will be checking the double points. We call a double point \textbf{safe} if its sink corner is not a sink disk. We will do multiple splittings (pushing arcs) to show that in the end we can make all double points safe. We break up our splitting procedure into several steps:
    ~\\

    \textbf{Step 1: Sink tube push on $\mathcal{A''}$.}

    Recall that to define pushing arcs, we need to regard some sectors as attached to a fixed surface, and identify their boundary branch locus as some curves on the fixed surface. Now consider how the $\beta$-disk sector is attached to the rest of $\mathcal{A''}$. Roughly speaking we can still regard it as attached to a torus $\Sigma^+$ along the $\beta$-curve, where outside the $(\alpha,\delta)$-source bigon $\Sigma^+$ consists of sectors coming from $\partial_h^+$, and at the $(\alpha,\delta)$-source bigon $\Sigma^+$ comes from sectors of $\mathcal{A}$ (see Figure~\ref{fig:21}).
    
    However, as one can observe in Figure~\ref{fig:21}, the torus $\Sigma^+$ is actually cut along $\delta$. That being said, $\Sigma^+$ actually has 2 boundary circles corresponding to $\delta$. One of them comes from the cusp of $\mathcal{A''}$, which we will denote as $\delta_{cusp}$ (see the green curve in Figure~\ref{fig:21}); the other is the $\delta$-boundary circle of $\partial_h^+$ formerly denoted as $l_{\delta}^+$, which we now denote as $\delta_{\partial_h}$. Also, since we have removed the $(\alpha,\beta)$-source bigon, the boundary arc of it is no longer part of the branch locus. So to push $\beta$-arcs on $\Sigma^+$, we need to meet the following conditions: 
    
    \begin{enumerate}
        \item we do not move the boundary arc of the $(\alpha,\beta)$-source bigon, and
        \item we do not push the arc across $\delta$.
    \end{enumerate}

    \textbf{Remark} We use $\Sigma^+$ to denote the ``torus'', because on the one hand the torus is much like our torus $\Sigma$ in the (1,1)-diagram, while on the other hand we will be only pushing the $\beta$-arc, so only considering one side of $\Sigma$.~\\

    %%this picture seems no longer needed
    %\begin{figure}safe
        %\begin{overpic}{pictures/cut_along_delta.pdf}
        %    \put(0.5,58.5){$\partial_h^+$}
        %    \put(20,63){$\beta$-disk}
        %    \put(32,58.5){$\delta$-cusp}
        %    \put(42,63){$\delta$-annulus}
        %    \put(67,66){boundary of $\delta$-annulus}
        %    \put(67,58.5){$\beta$-disk}
        %    \put(96,58.8){$\partial_h^+$}
        %    \put(53,44.4){$\partial_h^-$}
        %    \put(0.5,17){$\partial_h^+$}
        %    \put(32,16.5){$\delta$-cusp}
        %    \put(42,22.5){$\delta$-annulus}
        %    \put(96,18){$\mathcal{A}$}
        %    \put(42,4){$\partial_h^-$}
        %    \put(22,37){$(a)$ Away from $(\alpha,\delta)$-source bigon}
        %    \put(30,-2){$(b)$ At $(\alpha,\delta)$-source bigon}
        %\end{overpic}
        %\caption{Attachment of the $\beta$-disk to %``$\Sigma$'' near $\delta$}
    %    \label{fig:a}
    %\end{figure}

    If there is only 1 $\beta$-arc in the $(\alpha,\delta)$-source bigon, then there is no double point besides $M,N$ (vertices of the $(\alpha,\beta)$-source bigon are not double points), and the branched surface $\mathcal{A''}$ is trivially sink disk free. Suppose there are multiple alternating $\beta$-arcs. Since the $\beta$-arcs are alternating, the $(\alpha,\beta)$-source bigon is adjacent to a $\beta$-sink sector (see Figure~\ref{fig:21}). Now we can do a ``sink tube push'' for $\beta$, according to the information of the $(\alpha,\beta)$-(1,1)-diagram (notice that the operation does not cross $\delta$, which only passes the $\beta$-parallel sectors of the $(\alpha,\beta)$-diagram and the octagon or hexagons). 

    \begin{figure}[!hbt]
        \begin{overpic}[scale=0.33]{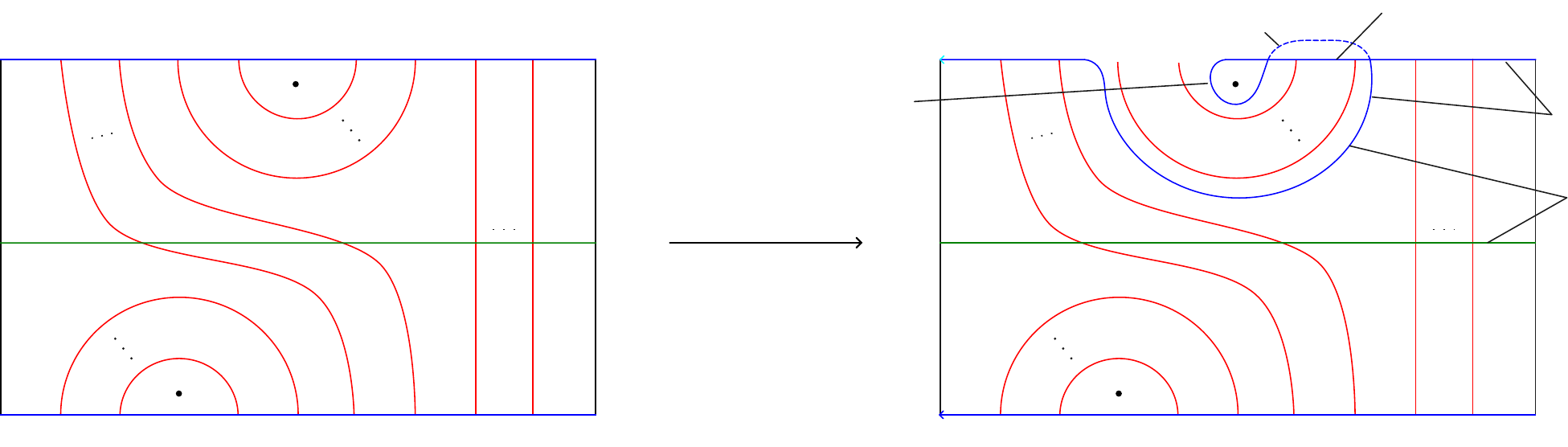}
            \put(10,3){\tiny$z$}
            \put(19,23){\tiny$w$}
            \put(36,0){\tiny $\color{ao}\beta$}
            \put(34.3,20){\tiny $\color{red}\alpha$}
            \put(36,11){\tiny $\color{green}\delta$}
            \put(100,16){\tiny parallel}
            \put(100,14){\tiny $\delta$-circles}
            \put(88.4,28){\tiny $\beta_c$}
            \put(99,20){\tiny $\beta_{\delta}$}
            \put(56,21){\tiny $\beta_s$}
            \put(54,27){\tiny arc pushed onto the $\beta$-disk}
        \end{overpic}
        \caption{($\beta$-)sink tube push for $\beta$ in standard position}
        \label{fig:22}
    \end{figure}
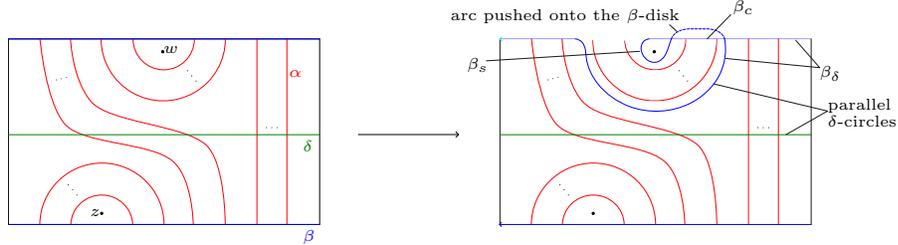
    
    The part of the $\beta$-curve remaining on $\Sigma^+$ (that is not pushed onto the $\beta$-disk) consists of two loops (each with a double point) and an arc connecting them. See Figure~\ref{fig:22} where we place $\beta$ in standard position instead. We denote the loop in the sink bigon $\beta_s$, the other loop $\beta_{\delta}$, and the connecting arc $\beta_c$. Notice that as shown in Figure~\ref{fig:22}, $\beta_{\delta}$ is isotopic to $\delta$ in the 2-pointed torus $(\Sigma,z,w)$, as they both intersect each vertical $\alpha$-arc exactly once. Hence on $\Sigma^+$ our $\beta_{\delta}$ is actually isotopic to one of the $\delta$-boundary components of $\Sigma^+$.

    The sink tube push produces two new double points $X_s$ on $\beta_s$ and $X_{\delta}$ on $\beta_{\delta}$. The sink corner of $X_s$ is an annulus because of the hole in the sink bigon. For $X_{\delta}$, we first notice that, $\beta_{\delta}$ can be isotoped in its branch direction to one of the $\delta$-boundary circles of $\Sigma^+$, either $\delta_{cusp}$ or $\delta_{\partial_h}$. If $\beta_{\delta}$ can be isotoped to $\delta_{\partial_h}$, the sink corner of $X_{\delta}$ is an annulus bounded by these two loops. If $\beta_{\delta}$ can be isotoped to $\delta_{cusp}$, the sink corner of $X_{\delta}$ (now a disk with some boundary $\alpha$-arcs) has a boundary arc from $\delta_{cusp}$, where the branch direction necessarily points outwards. Hence the two new double points are safe, and possible sink disks come from (sink corners of) double points at the intersection of $\beta_c$ and the boundary $\alpha$-arc of the $(\alpha,\delta)$-source bigon (i.e. the double points of $\mathcal{A''}$ that remains on $\Sigma^+$ after the sink tube push).

    \textbf{Remark} Whether $\beta_{\delta}$ can be isotoped to $\delta_{cusp}$ or $\delta_{\partial_h}$ depends on the parity of the number of alternating $\beta$-arcs inside the $(\alpha,\delta)$-source bigon. For example, in Figure~\ref{fig:21}, there is an even number of alternating $\beta$-arcs, so the branch direction of the outer $\beta$-arc containing $P$ is pointing towards $\delta$ (actually $\delta_{\partial_h}$), and this arc later becomes a part of $\beta_{\delta}$. If the number is odd, then the $\beta$-arc inside the $(\alpha,\delta)$-source bigon that is closest to the boundary $\delta$-arc would have branch direction pointing towards $\delta_{cusp}$, and would later become a part of $\beta_{\delta}$ instead. 
    ~\\

    \textbf{Step 2: Reduce $\beta_c$ to a rational tangle strand on a 4-basepoint sphere}

    \begin{figure}[!ht]
        \begin{overpic}[scale=0.65]{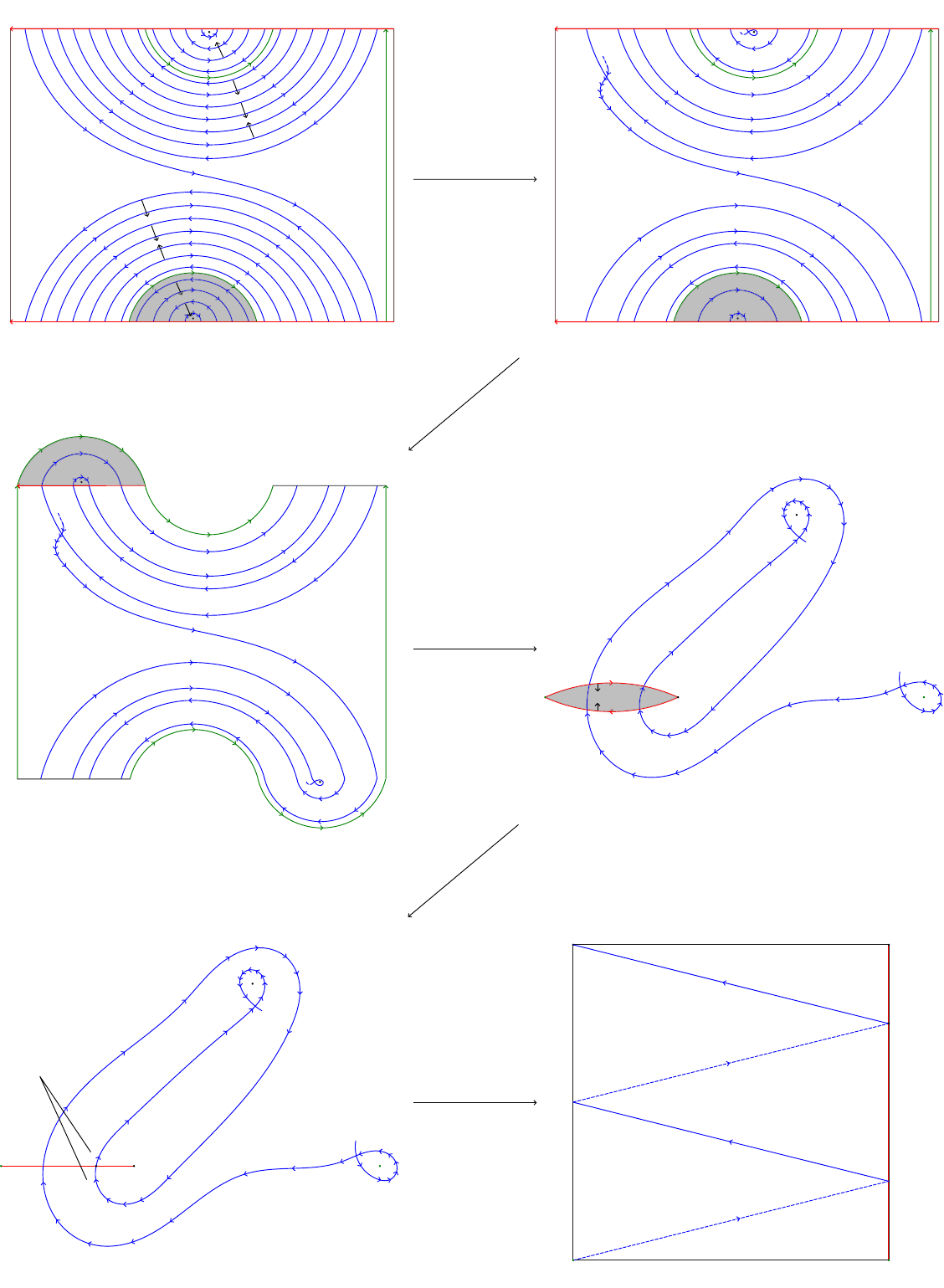}
            \put(9.5,72){\tiny knot (23,11,1,7)}
            \put(29,74.2){\Tiny $\color{red} \alpha$}
            \put(25,83){\Tiny $\color{ao} \beta$}
            \put(29,90){\Tiny $\color{green} \delta$}
            \put(9.4,74){\Tiny $M$}
            \put(19.3,74){\Tiny $N$}
            \put(29.3,98){\Tiny $M$}
            \put(10.6,98){\Tiny $N$}
            \put(14.6,74.2){\Tiny $z$}
            \put(15.6,97.9){\Tiny $w$}
            \put(31.1,87){\tiny $\beta$-sink tube push}
            \put(51.7,74){\Tiny $M$}
            \put(61.6,74){\Tiny $N$}
            \put(71.7,98){\Tiny $M$}
            \put(53.7,98){\Tiny $N$}
            \put(56.9,74.2){\Tiny $z$}
            \put(58,97.9){\Tiny $w$}
            \put(37,68){\tiny cut along $\delta$}
            \put(6,61.6){\Tiny $z(a)$}
            \put(25,39.5){\Tiny $w(d)$}
            \put(1.5,50){\Tiny $\color{green} \delta_{cusp}(b)$}
            \put(30.1,45){\Tiny $\color{green} \delta_{\partial_h}(c)$}
            \put(30.8,51){\tiny collapse circle boundaries}
            \put(61.6,59.2){\Tiny $d$}
            \put(53,46){\Tiny $a$}
            \put(41.4,45.3){\Tiny $\color{green} b$}
            \put(70.7,45.6){\Tiny $\color{green} c$}
            \put(37,32){\tiny collapse $\alpha$-bigon}
            \put(19.6,23){\Tiny $d$}
            \put(10.5,9.5){\Tiny $a$}
            \put(0,9.7){\Tiny $\color{green} b$}
            \put(28.7,9.3){\Tiny $\color{green} c$}
            \put(7.6,8.6){\Tiny $e$}
            \put(-3,17){\Tiny sink corners at $e$}
            \put(27,17.3){\tiny collapsing $\beta$-loops and}
            \put(27,16){\tiny put in standard position}
            \put(43.6,27){\Tiny $d$}
            \put(69,27){\Tiny $a$}
            \put(69.3,2){\Tiny $\color{green} b$}
            \put(43.5,2){\Tiny $\color{green} c$}
            \put(69.2,20.5){\Tiny $e$}
            \put(47.5,0){\tiny strand of rational tangle $\frac{1}{4}$}
        \end{overpic}
        \caption{Rational tangle strand from knot (23,11,1,7)}
        \label{fig:23}
    \end{figure}
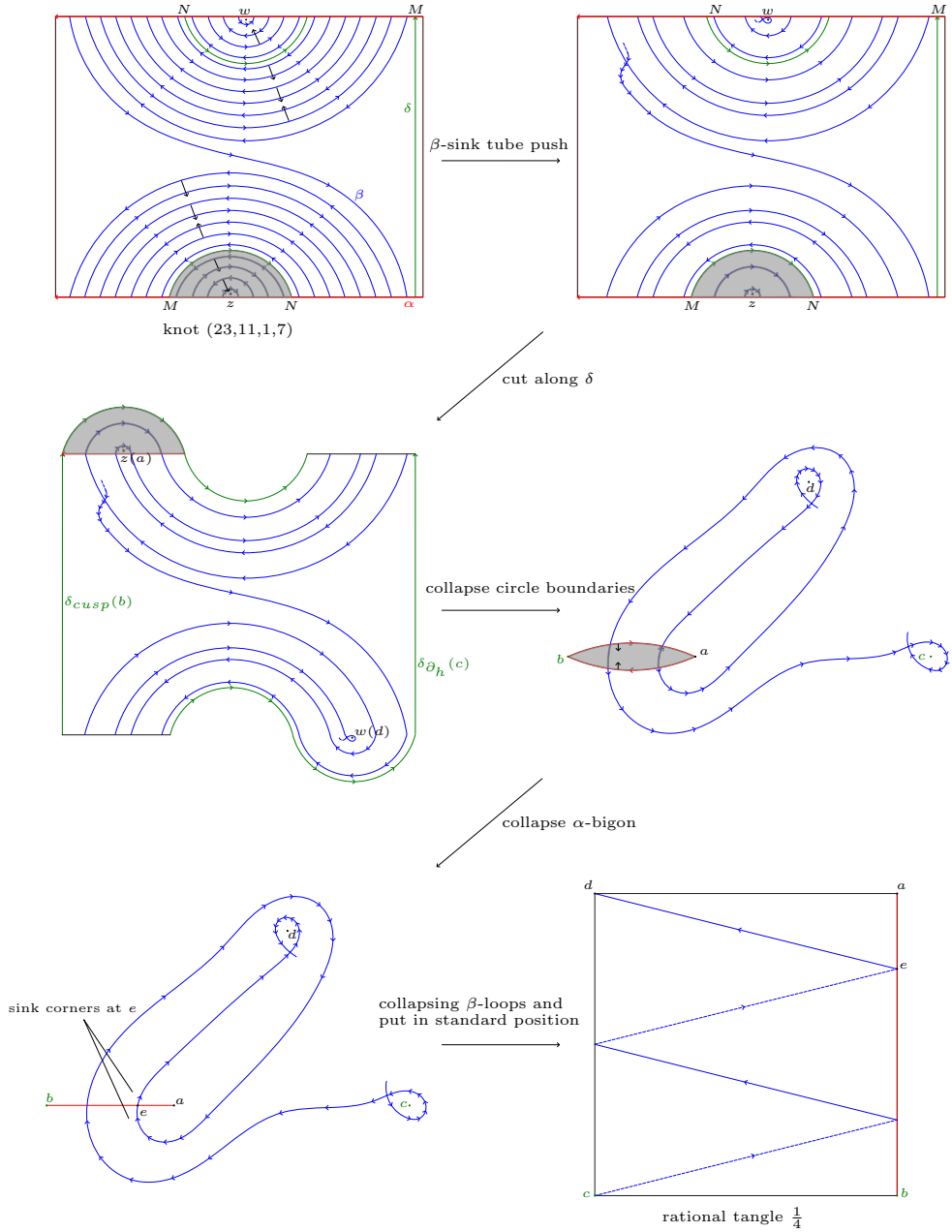

    We regret to say that Step 1 is not enough: there might still be sink disks left. To get a better characterization of how $\beta_c$ intersects the $\alpha$-arc, and also to better describe our further splittings, we need to modify our model. We remark that the following ``collapsing'' operations are purely for simplifying our characterization, and do not actually happen on our branched surface.

    Recall that our $\Sigma^+$ is actually cut along $\delta$, with two boundary $\delta$-circles $\delta_{cusp}$ and $\delta_{\partial_h}$. Moreover, there is also a boundary circle in the sink bigon bounding the hole. Since we can never push arcs across these boundary circles, we can collapse(quotient) them to points. If we further mark the midpoint of the $\alpha$-arc in the $(\alpha,\beta)$-source bigon (which also cannot be moved by our pushing arcs operations), we get a sphere with 4 basepoints marked. We denote it by $(S,a,b,c,d)$, where $S$ is the sphere we get, $a$ is the midpoint of the $\alpha$-arc in source bigon, $b$ is the $\delta_{cusp}$ boundary circle, $c$ is the $\delta_{\partial_h}$ circle, and $d$ is the boundary circle in the sink bigon, see the first four pictures of Figure~\ref{fig:23}, where we use the knot (23,11,1,7) as an example and $\alpha$ is placed in standard position.
    
    Now as in the fourth picture of Figure~\ref{fig:23}, the $(\alpha,\delta)$-source bigon (see the shaded regions) becomes a bigon connecting two basepoints on $S$, bounded only by $\alpha$-arcs (that come from the boundary $\alpha$-arc of the source bigon). We can then collapse it to a \textit{single} arc connecting the basepoints $a,b$. More precisely, the bigon can be parametrized by the unit disk, so that $(\pm 1,0)$ correspond to the vertices $a,b$ and the $\beta$ arcs intersecting the bigon are vertical; by collapsing we mean a projection of the unit disk to the $x$-axis. See this procedure in Figure~\ref{fig:23} from the fourth picture to the fifth picture.
    
    Recall again that after the sink tube push the $\beta$-arcs left on ``$\Sigma$'' are two loops $\beta_s$, $\beta_{\delta}$ and a connecting arc $\beta_c$. Now in our 4-point sphere $(S,a,b,c,d)$, each of the two $\beta$-loops bounds a small disk with a basepoint inside, and its branch direction points into the disk. We then further collapse(quotient) the loops and the disks they bound to the basepoints inside. Then $\beta_c$ becomes an arc connecting the basepoint $d$ to either $b$ or $c$, and thus may be viewed as a strand of a rational tangle. See the last two pictures in Figure~\ref{fig:23}.

    We fix a frame so that $\alpha$ represents a strand of the rational tangle $\frac{1}{0}$. Since $\alpha,\beta$ come from a reduced (1,1)-diagram, their position on $S$ is tight, i.e. having no trivial bigons. Then $\beta_c$ is actually \textit{characterized by} a rational number $r=\frac{p}{q}$, where $p,q\geq 0$, $(p,q)=1$. In fact, we can place $\beta_c$ in certain ``standard position'' as in Figure~\ref{fig:23} bottom-right or Figure~\ref{fig:24}.$(a)$, where the rectangle is bubbled to represent the sphere $S$. The boundary segments of the bubbled rectangle are chosen to represent rational tangles $\frac{1}{0}$ and $\frac{0}{1}$. The solid $\beta$-segments are in the front of the bubbled rectangle, while the dashed segments are on the back of it. If we regard the rectangle as the unit rectangle, then these segments are all of slopes $\pm \frac{p}{q}$. We also recall that the numbers $p,q$ actually count the total intersections of $\beta_c$ with the $\frac{0}{1}$ and $\frac{1}{0}$ arcs $bc$ and $ab$, respectively, where we count 2 for each intersection other than the basepoints, and 1 for each intersection at the basepoints. %By carefully choosing the frame, we can further assume $0\leq r\leq 1$ (we keep 1 for technical reasons). 

    We claim that our collapsings do not collapse possible sink disks. In fact, when quotienting the $\beta$-loops to points we only quotient the disk with basepoint in the branch direction of each of the $\beta$-loops. Since the basepoint is collapsed from a circle, the collapsed regions are all bounded by the circle. However, back on the branched surface the circle is either a boundary circle, or $\delta_{cusp}$ bounding the collapsed regions opposite its branch direction. Hence by collapsing loops we are not collapsing any sink disk. On the other hand, the bigon is bounded by $\alpha$ opposite its branch direction, so we do not collapse any sink corner. Hence no sink disk is collapsed by our operations. However, we do collapse the double points in pairs (that are originally connected by $\beta$-arcs in the $(\alpha,\delta)$-source bigon, see $Q,R$ in Figure~\ref{fig:21} for an example), so in the 4-pointed sphere after collapsings each intersection of $(\alpha,\beta)$ corresponds to two sink corners, see the fifth picture of Figure~\ref{fig:23}.
    ~\\

    \textbf{Step 3: Set up ``sink tube push'' for the rational tangle strand}

    We now spot possible sink disks on $(S,a,b,c,d)$ after collapsings. Notice that now for any sector on $S$ with a boundary $\alpha$-arc, its boundary $\alpha$-arc would have branch direction pointing inwards (we have collapsed the other side of $\alpha$). So to check if a sector on $S$ is a sink disk, we only need to check the branch direction of its boundary $\beta$-arcs. See the shaded region in Figure~\ref{fig:24}.$(a)$ for an example of sink disks here. We still need to push $\beta$-arcs on $S$ to eliminate such sink disks. Now that $S$ is obtained by cutting $\Sigma^+$ along $\delta$ and collapsing, when pushing $\beta$-arcs on $S$ we automatically will not push across $\delta$. Hence we only need to guarantee that we do not move the $\beta$-boundary arc of the $(\alpha,\beta)$-source bigon. This arc, after the collapsings, is now the point $e$ as shown in Figure~\ref{fig:23} and Figure~\ref{fig:24}.$(a)$. It is the $(\alpha,\beta)$-intersection that is closest to $a$. In our following pushings, we will treat it as an anchor for $\beta_c$, since it can never be moved.

    \begin{figure}[!hbt]
        \begin{overpic}[scale=0.5]{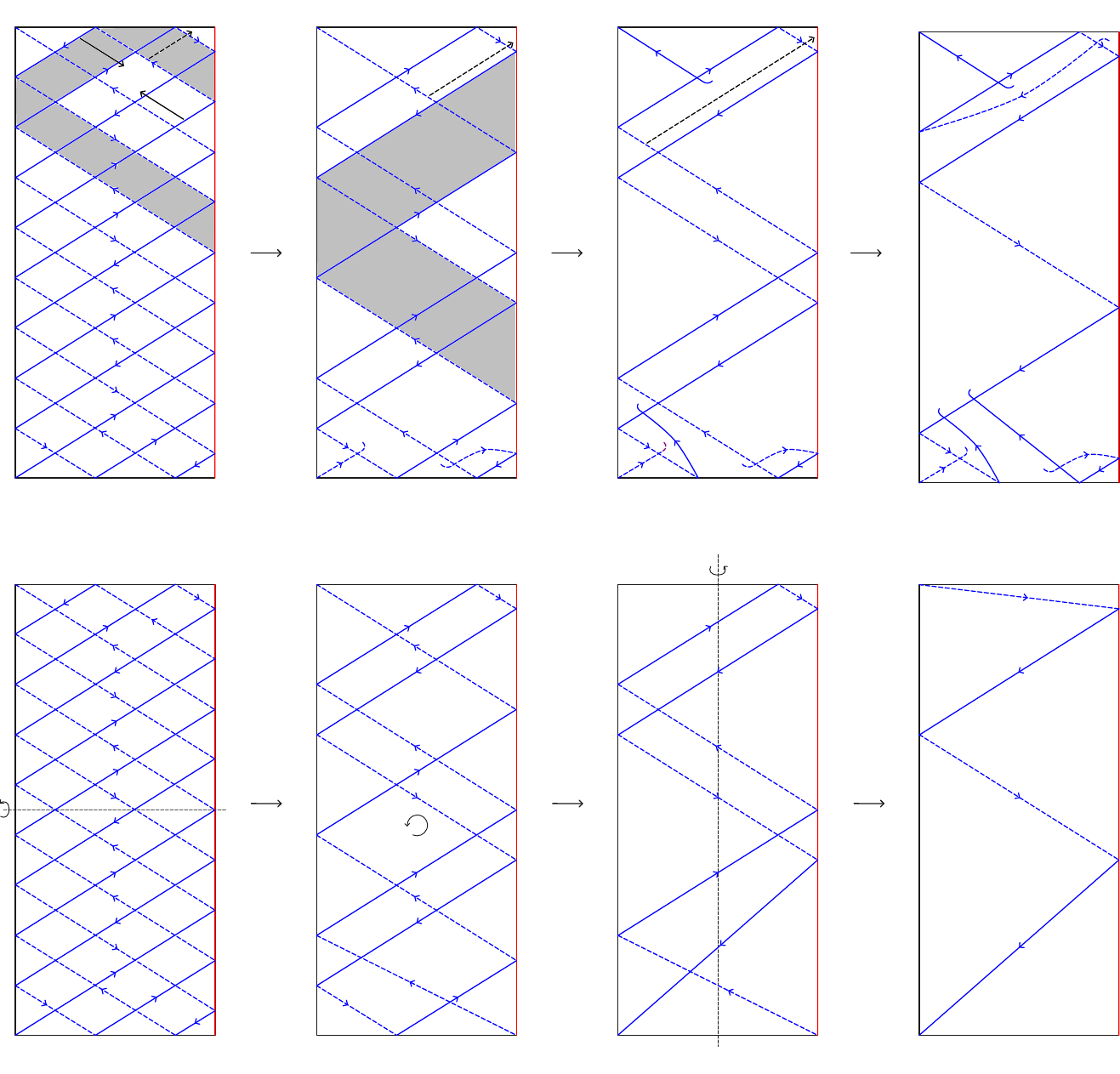}
            \put(8,50){($a$)}
            \put(35.5,49.7){($b$)}
            \put(62.2,50){($c$)}
            \put(89,49.5){($d$)}
            \put(8,0){($e$)}
            \put(35.5,0){($f$)}
            \put(62.2,0){($g$)}
            \put(89,-0.3){($h$)}
            \put(19.5,93){\tiny$a$}
            \put(19.5,53){\tiny$b$}
            \put(0,53){\tiny$c$}
            \put(0,93){\tiny$d$}
            \put(19.5,91){\tiny$e$}
            \put(-0.8,3){\tiny$P$}
            \put(15.8,41.1){\tiny$R$}
            \put(15.7,5){\tiny$Q$}
            \put(-0.5,43){\tiny$S$}
            \put(8.5,-3){\tiny$\frac{5}{18}$}
            \put(36,-3){\tiny$\frac{3}{11}$}
            \put(63,-3){\tiny$\frac{2}{7}$}
            \put(90,-3){\tiny$\frac{1}{4}$}
            %\put(42,2){\Tiny$\frac{1}{p}(\frac{1}{5})$}
            \put(-2,22){\Tiny $\tau_1$}
            \put(34,22){\Tiny $\tau_2$}
            \put(61,45){\Tiny $\tau_3$}
        \end{overpic}
        \caption{Reducing rational tangles}
        \label{fig:24}
    \end{figure}
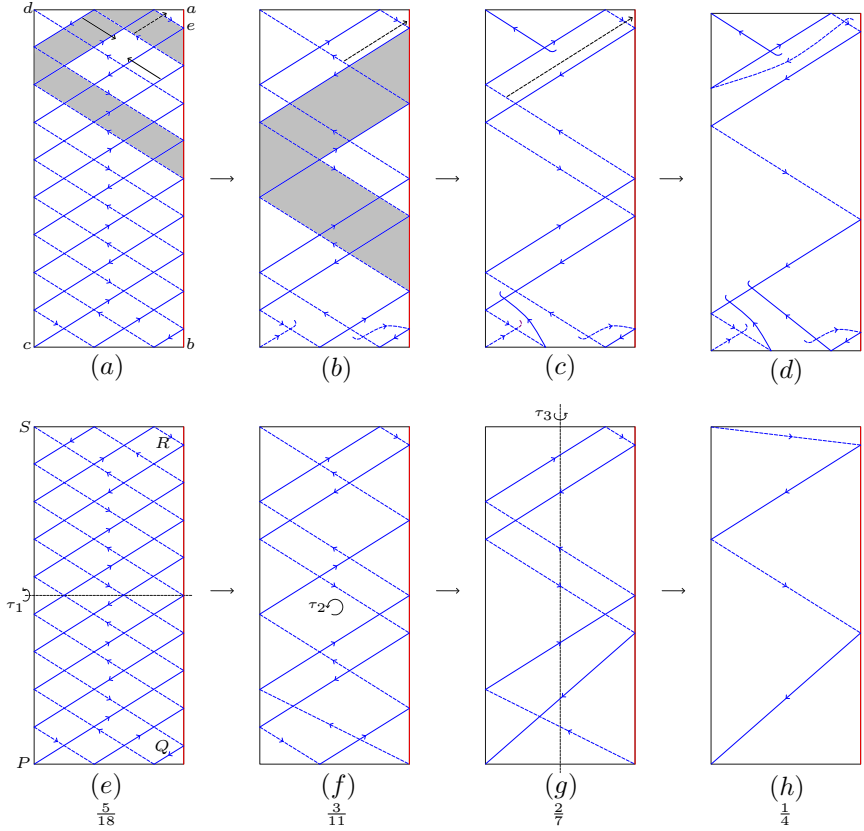

    Since the $\alpha$-arc does intersect $\beta_c$ at $e$, which is not a basepoint, we know $q\geq 2$, and thus $p\geq 1$. \textbf{We claim that if $p=1$, then there is already no sink disk.} In fact, suppose we start tracing the $\beta_c$-arc from $d$; then since it will not intersect horizontal segments $ad$ and $bc$ again before getting to the endpoint, it will necessarily wind from $d$ all the way down to the bottom, see Figure~\ref{fig:23} bottom-right. Hence $\beta_c$ always intersects $\alpha$ in the same direction. Now consider the segments of the $\alpha$-arc cut out by $\beta$. We claim that none of them is the boundary arc of a sink disk. In fact, except the segments with an endpoint $a$ or $b$, all other segments are ``$\beta$-parallel'' in the sense of Definition~\ref{def:1}, thus cannot appear on the boundary of a sink disk. For the segment with an endpoint $b$, we notice that back on $\Sigma^+$ it is attached to a circle at $b$. According to our construction, this circle is either the boundary circle $\delta_{\partial_h}$, or a branch locus circle, whose branch direction points to the other side (i.e. not the side our $\alpha$-segment is attached to). Either way our $\alpha$-segment cannot be a boundary arc of a sink disk. The segment with endpoint $a$ is $ae$ coming from the removed source bigon, so actually is not part of the branch locus. It follows that no sink corners at the $(\alpha,\beta)$-intersections are actually sink disks, and hence our branched surface (obtained by doing a sink tube push to $\mathcal{A''}$) is already sink disk free. \textbf{From now on we suppose $p\geq 2$.} %Then $q>p\geq 2$ since $r\leq 1$ and $(p,q)=1$.

    We show that, for $\beta_c$ in standard position on $(S,a,b,c,d)$ we can again define a ``sink tube'' to push. We can regard the boundary of the rectangle $abcd$ as a simple closed curve on $S$. We call this curve ($a-b-c-d-a$) the (smoothed) \textbf{frame circle} and denote it as $\epsilon$. Now $\beta_c$ and $\epsilon$ together cut $S$ into pieces (we remark that vertices of these pieces are the intersection points of $\beta_c$ and $\epsilon$, so basepoints are not necessarily vertices). Since the $\beta$-arc is oriented, we can define $\beta$-sink, source, and parallel arcs as in Definition~\ref{def:1} for the $\epsilon$-arcs cut out by $\beta$. And then we can define $\beta$-sink, source, and parallel sectors for the quadrilateral pieces of $S$ cut out by $\beta_c$ and $\epsilon$. In fact, if we first cut along $\epsilon$ and then along $\beta_c$, we can see that we first cut $S$ into two disks, and then cut the disks with parallel $\beta$-arcs; hence we're getting mostly quadrilateral pieces, except for certain ``bigons'' near the basepoints, see Figure~\ref{fig:24}.$(a)$. (More precisely, for each of the bigons obtained by cutting disks with parallel $\beta$-arcs, we see a basepoint in the interior of its $\epsilon$-boundary; we say this bigon is \textbf{near} the basepoint.)
    
    There are two basepoints that are endpoints of $\beta$ and two basepoints that are not. For each of the two non-endpoints, we notice that the $\epsilon$-arc containing them does not contain an entire edge of the rectangle (if it does, then $\beta_c$ would intersect that edge only once at the basepoint, contradicting the fact that $p,q\geq 2$). It follows that the $\epsilon$-arc appears on the boundary of a bigon on one side and a $\beta$-sink or source sector on the other, see Figure~\ref{fig:25}.$(a)$, where we smoothed the bubbled rectangle near the basepoints. Still, we call the bigon a sink (resp. source) bigon if the boundary $\epsilon$-arc is $\beta$-sink (resp. source). 

    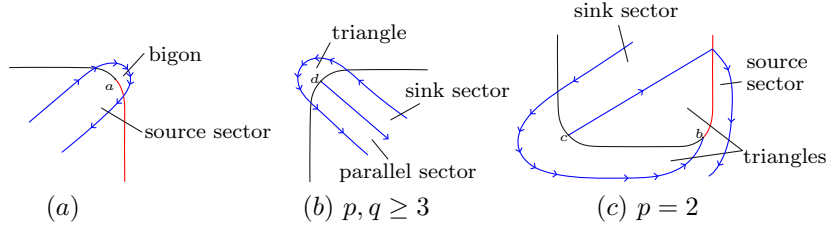
\begin{figure}[!hbt]
        \begin{overpic}[scale=0.9]{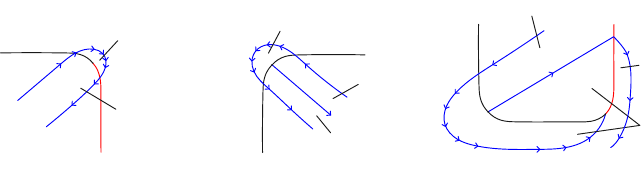}
            \put(5,-2){$(a)$}
            \put(40,-2){$(b)$ $p,q\geq 3$}
            \put(80,-2){$(c)$ $p=2$}
            \put(13,15){\Tiny $a$}
            \put(18.5,8.5){\Small source sector}
            \put(19,19){\Small bigon}
            \put(41,16){\Tiny $d$}
            \put(54,14){\Small sink sector}
            \put(45,3){\Small parallel sector}
            \put(44,22){\Small triangle}
            \put(75,8){\Tiny $c$}
            \put(93.5,8.5){\Tiny $b$}
            \put(77,25){\Small sink sector}
            \put(100.5,18){\Small source}
            \put(100.5,15){\Small sector}
            \put(100,5){\Small triangles}
        \end{overpic}
        \caption{Bigons and Triangles at the basepoints}
        \label{fig:25}
    \end{figure}
    
    On the other hand, for each of the two endpoints, the ``bigon'' \textit{near} it is actually a triangle since the endpoint itself is also a vertex, see Figure~\ref{fig:25}.$(b)(c)$. When both $p,q\geq 3$, neither of the two $\epsilon$-edges of the triangle is a whole edge of the rectangle $abcd$, so they each bound a quadrilateral on the other side. Moreover, as shown in Figure~\ref{fig:25}.$(b)$, one of them is a parallel sector and the other is a sink or source sector. When $p=2$ or $q=2$ (notice it cannot be both since $(p,q)=1$), the two triangles corresponding to endpoints share an $\epsilon$-edge, which is actually an entire edge of the rectangle $abcd$. Moreover, since that $\beta_c$ has $(p+q-2)$ intersections with $\epsilon$ other than the endpoints, and that exactly one of $p,q$ is even (they are coprime and one equals 2 by our case supposition), away from endpoints $\beta_c$ intersects $\epsilon$ an odd number of times. Hence near the endpoints the $\beta_c$-threads are on different sides of $\epsilon$, see Figure~\ref{fig:25}.$(c)$. Now by our definition the shared $\epsilon$-edge of the two triangles is $\beta$-parallel, and the two quadrilaterals bounded by the not-shared $\epsilon$-edges of the two triangles are exactly one sink sector and one source sector, as in Figure~\ref{fig:25}.$(c)$.

    Recall that in the (1,1)-diagram case, we are able the define the sink tube since there are only two $\beta$-sink arcs among the boundary $\alpha$-arcs of non-quadrilateral pieces (so that one is eventually connected to the other by pasting $\beta$-sink sectors). To establish a similar construction for $(S,\beta_c,\epsilon)$, we need to verify the same thing, i.e. there are only two $\beta$-sink arcs among the boundary $\epsilon$-arcs of the bigons and triangles. We achieve this by introducing symmetries of $\beta_c$ on $(S,a,b,c,d)$. One should think of this as an analogue to the hyperelliptic involution in the (1,1)-diagram case.

    There are $\pi$-rotations of $S$ that keep $\epsilon$ invariant, while interchanging the basepoints in pairs. We regard $(S,a,b,c,d)$ as the bubbled rectangle, and define $\tau_1$, $\tau_2$, and $\tau_3$ as in Figure~\ref{fig:24}.$(e)$, $(f)$, and $(g)$, respectively. More precisely, $\tau_1$ and $\tau_3$ interchange the front rectangle and the back rectangle, and then on each rectangle act as a reflection in the axis; $\tau_2$ keeps the front and back rectangles invariant, while acting as a $\pi$-rotation or central reflection on each rectangle. We see that $\tau_1$ interchanges the basepoints $a,b$ and $c,d$, $\tau_2$ interchanges $a,c$ and $b,d$, and $\tau_3$ interchanges $a,d$ and $b,c$.
    ~\\

    \textbf{Claim.} If the endpoints of $\beta_c$ are interchanged by $\tau_{i_0}$, then $\tau_{i_0}(\beta_c)$ is isotopic to $\beta_c$ on $(S,a,b,c,d)$, but has its orientation reversed.

    \begin{proof}[Proof of the claim]
        Since $\beta_c$ and $\epsilon$ are in tight position on $(S,a,b,c,d)$, we know $\tau_{i_0}(\beta_c)$ and $\tau_{i_0}(\epsilon)=\epsilon$ are also in tight position. Hence $\tau_{i_0}(\beta_c)$ is also characterized by a rational number $\frac{p'}{q'}$, $p',q'\geq 0$ and $(p',q')=1$. Now since $\tau_{i_0}$ interchanges the basepoints, it sends the rectangle edge $ab$ to either itself or $cd$. Let $\mathrm{IC}(\beta_c,ab)$ be the intersection count of $\beta_c$ and $ab$, where we recall we count 1 for each basepoint intersection and 2 for each non-basepoint intersection. Noticing that $\mathrm{IC}(\tau_{i_0}(\beta_c),ab)=\mathrm{IC}(\tau_{i_0}(\beta_c),cd)=q'$, we have \[q=\mathrm{IC}(\beta_c,ab)=\mathrm{IC}(\tau_{i_0}(\beta_c),\tau_{i_0}(ab))=\mathrm{IC}(\tau_{i_0}(\beta_c),ab)=q'.\] Similarly we can show $p=p'$. Hence $\tau_{i_0}(\beta_c)$ is actually isotopic to $\beta_c$ on $(S,a,b,c,d)$. Moreover, since $\tau_{i_0}$ interchanged the endpoints of $\beta_c$, the starting point of our oriented $\tau_{i_0}(\beta_c)$ is actually the ending point of $\beta_c$, and hence the orientation is reversed.
    \end{proof}
    
    The claim gives us the symmetry of $\beta_c$ we need. According to the endpoints of our $\beta_c$, we can pick a suitable $\tau_{i_0}$ interchanging its endpoints. Now since $\tau_{i_0}$ does not fix any basepoint, it cannot fix any bigon. Since there are exactly two bigons, each is obtained from the other by applying the involution $\tau_{i_0}$ and then reversing the orientation of $\tau_{i_0}(\beta_c)$. The image of a $\beta$-sink (resp. $\beta$-source) bigon under $\tau_{i_0}$ is still $\beta$-sink (resp. $\beta$-source), but reversing the orientation of $\tau_{i_0}(\beta_c)$ makes it $\beta$-source (resp. $\beta$-sink). Hence there are exactly one sink bigon and one source bigon. Similarly, there are exactly one triangle with a $\beta$-sink $\epsilon$-edge and one triangle with a $\beta$-source $\epsilon$-edge (recall a triangle always contains a $\beta$-parallel $\epsilon$-edge). Hence for all the non-quadrilateral regions, there are only two boundary $\epsilon$-arcs that are $\beta$-sink(resp. source). It then follows that the two arcs must be connected by pasting $\beta$-sink(resp. source) sectors along the boundary $\epsilon$-arcs. Thus we can still define the \textbf{sink (resp. source) tube} here, and a statement parallel to Proposition~\ref{lem:2} still holds.

    We now introduce the ``sink tube push'' in this case. Recall that since we have already cut along $\delta$, when pushing the arcs we only need to make sure that we do not move the boundary $\beta$-arc of the source bigon, or the point $e$ after collapsings as in Figure~\ref{fig:24}.$(a)$. We used to guarantee this by manually spotting the source bigon at one of the long $\beta$-arcs of the sink tube, but for our rational tangles there is a more natural way. One can define each $\beta$-segment cut out by $\epsilon$ to be of length 1. By the previously discussed symmetry of $\beta_c$ we can describe it as Figure~\ref{fig:26}.$(a)$ upper-left, where $P,S$ are endpoints of $\beta_c$, $Q$ is the midpoint of the boundary $\beta$-arc of the sink bigon, and $R$ the midpoint of the boundary $\beta$-arc of the source bigon, s.t. $P,Q,R,S$ appear in order and $\mathrm{d}(P,Q)=\mathrm{d}(R,S)$. We then define our ``sink tube push'' to push the long $\beta$-arc of the sink tube in $PQ$ onto the other long $\beta$-arc in $QS$. Moreover, we modify our operation at the ends of the tube, so that near the $\beta$-sink bigon we obtain a $\beta$-loop bounding a basepoint (see bottom-right corner of Figure~\ref{fig:24}.$(a)(b)$, or Figure~\ref{fig:26}.$(b)$), and near the triangle with $\beta$-sink edge containing $P$ we push the $\beta$-arc connected to $P$ onto the $\beta$-edge of the triangle (see the bottom-left corner of Figure~\ref{fig:24}.$(a)(b)$, or the first arrow in Figure~\ref{fig:26}.$(c)$). Since $e$ is a vertex of the source bigon, $\mathrm{d}(e,R)=1/2$. Since $R$ is not on $PQ$, $e$ is also not on $PQ$. Thus our operation does not move $e$ and is hence a legal push.

    \begin{figure}[!hbt]
        \begin{overpic}[scale=0.5]{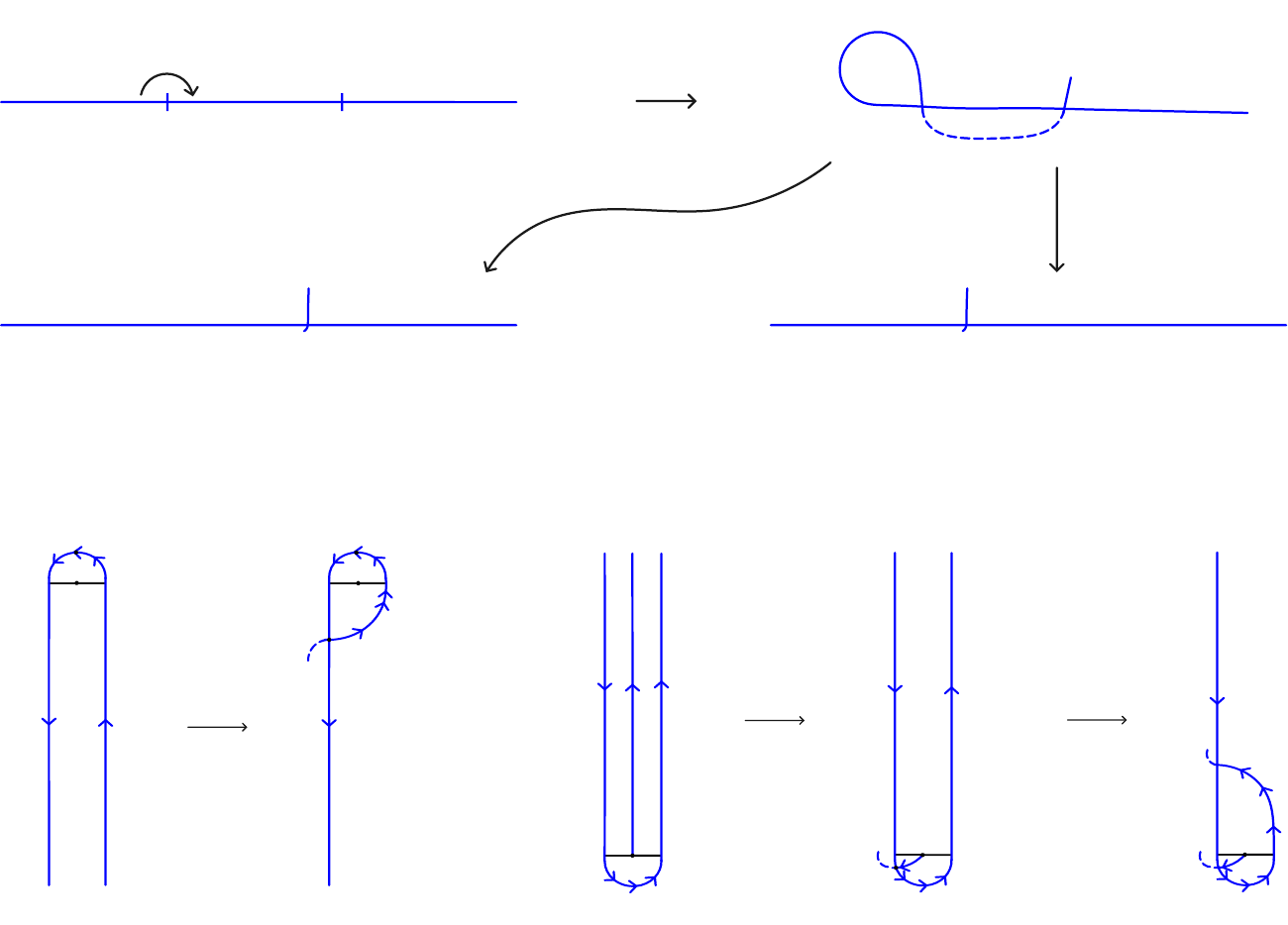}
            \put(50,36){$(a)$}
            \put(15,-1){$(b)$}
            \put(70,-1){$(c)$}
            \put(-1,62){\tiny$P$}
            \put(12,62){\tiny$Q$}
            \put(25.5,62){\tiny$R$}
            \put(39,62){\tiny$S$}
            \put(83,66){\tiny$P$}
            \put(66,62){\tiny$Q$}
            \put(97,62){\tiny$S$}
            \put(25,55){\tiny$|PQ|>|QR|$}
            \put(84,55){\tiny$|PQ|<|QR|$}
            \put(23,50){\tiny $P$}
            \put(-2,44.5){\tiny $Q(S')$}
            \put(10,44.5){\tiny $R(R')$}
            \put(21,44.5){\tiny $X_p(Q')$}
            \put(38,44.5){\tiny $S(P')$}
            \put(74,50){\tiny $P$}
            \put(57,44.5){\tiny $Q(P')$}
            \put(72,44.5){\tiny $X_p(Q')$}
            \put(84,44.5){\tiny $R(R')$}
            \put(97,44.5){\tiny $S(S')$}
            \put(49,6){\tiny $P$}
            \put(71,6){\tiny $P$}
            \put(67,3){\tiny $X_p$}
            \put(5,30){\tiny $Q$}
            \put(27,30){\tiny $Q$}
            \put(22.3,23){\tiny $X_q$}
        \end{overpic}
        \caption{Sink tube push for the rational tangle strand}
        \label{fig:26}
    \end{figure}

    We now show that the new double points created by our sink tube push for $\beta_c$ are all safe. There are two new double points: $X_q$ on the new $\beta$-loop $\beta_q$ near $Q$, and $X_p$ on the $\beta$-edge of the triangle near $P$, which is connected to $P$ by a short arc in the triangle, see Figure~\ref{fig:26}.$(b)(c)$. Every sink corner of possible double points on $\beta_q$ has the basepoint inside on its boundary (there might be double points other than $X_q$ if the basepoint is $b$, where $\alpha$ intersects $\beta_q$). Since $a$ is on the boundary of the source bigon (of $(S,\beta_c,\epsilon)$), the basepoint inside $\beta_q$ cannot be $a$. Hence it is obtained by collapsing a circle. Again, back on $\Sigma^+$ the circle before collapsing is either a boundary circle of the branched surface or a branch locus circle whose branch direction points to the other side. Sink corners of double points on $\beta_q$ all have part of this circle on their boundaries, thus cannot be sink disks. The sink corner of $X_p$ has a boundary arc $PX_p$, and thus has the basepoint $P$ on its boundary. Now since $P$ is an endpoint of $\beta_c$, it is not $a$, and is hence obtained by collapsing a $\beta$-loop. Yet again, back on $\Sigma^+$ the sink corner of $X_p$ has a boundary arc belonging to this collapsed loop, where the branch direction points to the other side. It follows that $X_p$ is also safe. 
    
    We do a further collapsing to prepare for the next step. Since double points on $\beta_q$ are all safe, we can further collapse (or quotient) $\beta_q$ and the disk it bounds in its branch direction on $S$ to the basepoint inside, without destroying possible sink disks. See Figure~\ref{fig:26}.$(a)$. For the purpose of the next step, we describe a \textbf{complete} ``sink tube push'' procedure for a rational tangle strand to include both the arc-pushing and the collapsing of the loop afterwards.
    ~\\

    \textbf{Step 4: Perform sink tube push repeatedly to eliminate sink disks}
    
    Unfortunately, only one sink tube push for $\beta_c$ is unlikely to eliminate all sink disks, see the shaded region of Figure~\ref{fig:24}.$(b)$ for a sink disk that remains after the sink tube push. Our strategy is to perform sink tube pushes repeatedly until we can guarantee that there is no sink disk left. To do this we need to obtain a new $\beta$-arc connecting basepoints.
    
    Now if we temporarily ignore the short $\beta$-segment $PX_p$ (we call it a \textbf{tail}), we get a new $\beta$-arc $\beta_c'$ connecting $Q$ and $S$ (see Figure~\ref{fig:26}.$(a)$ upper-right, notice the loop at $Q$ is collapsed). This $\beta_c'$ is represented by some rational number $\frac{s}{t}$, where $s,t\geq 0$, $(s,t)=1$ and $s+t<p+q$ (the number of intersections is decreasing at least by 1 by collapsing the $\beta$-loop at $Q$ to the basepoint). Moreover, if $s,t\geq 2$, we still can define the sink tube for $(S,\beta_c',\epsilon)$. Since we ignore the tail $PX_p$, $P$ is no longer a vertex, and thus the triangle of $(S,\beta_c,\epsilon)$ at $P$ turns to a bigon of $(S,\beta_c',\epsilon)$; in fact, the triangle with a \textbf{sink} $\epsilon$-edge becomes the \textbf{sink} bigon of $(S,\beta_c',\epsilon)$, as shown in the middle picture of Figure~\ref{fig:26}.$(c)$. On the other hand, the source bigon of $(S,\beta_c',\epsilon)$ is the same as the source bigon of $(S,\beta_c,\epsilon)$ at $R$. Depending on whether $\mathrm{d}(P,Q)>\mathrm{d}(Q,R)$ or $\mathrm{d}(P,Q)<\mathrm{d}(Q,R)$ (they cannot be equal since $\mathrm{d}(P,Q)$ contains a half length at $Q$ while $\mathrm{d}(Q,R)$ contains two halves, one at each end, and is thus an integer), we can set new $P',Q',R',S'$ as in the bottom two pictures of Figure~\ref{fig:26}.$(a)$ (we remark that one can identify $Q'$, the midpoint of the boundary $\beta$-arc of the new sink bigon, with $X_p$). We can then do a ``sink tube push'' for $\beta_c'$, pushing our new $P'Q'$ onto $Q'S'$, in the same sense as before. Notice that since the tail $PX_p$ is inside the new sink bigon, it is not moved or moved over by our new push, see the third picture of Figure~\ref{fig:26}.$(c)$ (hence we are good to first ignore the tail and define a sink tube push with $\beta_c'$).

    It is still true that the new double points created by pushing $\beta_c'$ are safe. We still get two new double points: $X_{q}'$ on the new loop $\beta_q'$ near $Q'$, and $X_p'$ connected to $P'$ by another short arc. In fact, the only difference from the $\beta_c$ case is the tail inside $\beta_q'$. However, with or without the tail, all sectors inside the small disk bounded by $\beta_q'$ in its branch direction have the basepoint $P$ on their boundaries. Again and again, back on $\Sigma^+$ these sectors have boundary arcs on a branch locus circle collapsed to $P$, whose branch direction points outwards to the other side, see Figure~\ref{fig:27}. It follows that the sink corners of possible double points on $\beta_q'$ cannot be sink disks; in particular $X_q'$ is safe. By essentially the same argument as the $\beta_c$ case we can show that $X_p'$ is also safe.

    Since in fact all double points on $\beta_q'$ are safe, we can still collapse this loop and the small disk it bounds to the basepoint $P$ inside, without collapsing sink disks. This way we successfully defined another \textit{complete} sink tube push procedure.

    \begin{figure}[!hbt]
        \begin{overpic}[scale=0.4]{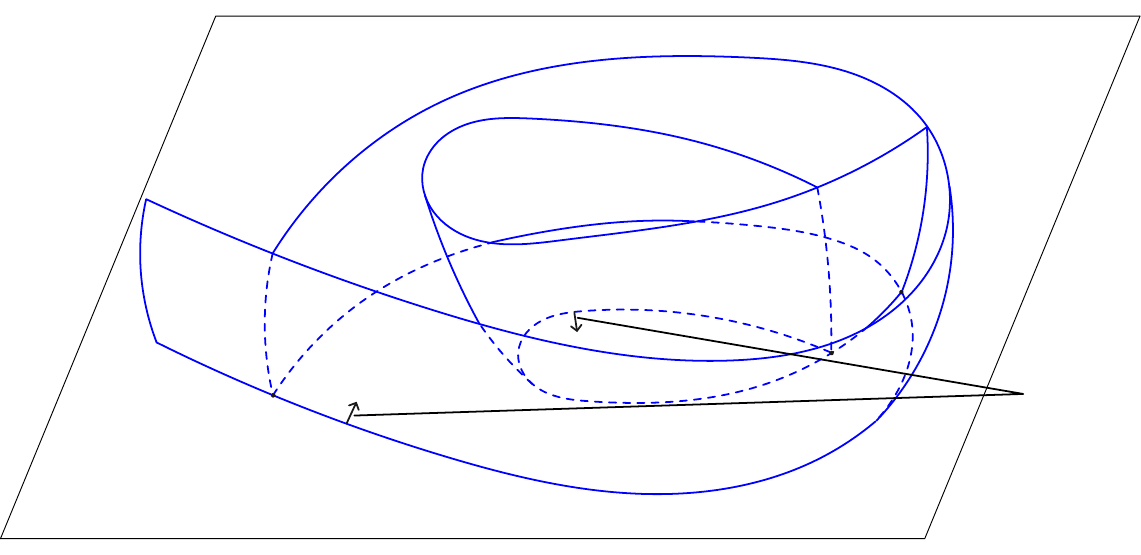}
            \put(90,13){\Small branch directions}
        \end{overpic}
        \caption{Nested $\beta$-circles to be collapsed}
        \label{fig:27}
    \end{figure}
    
    If we ignore the tail from the new $\beta_c'$-sink tube push, we get yet another $\beta$-arc connecting basepoints. It follows that we can \textit{repeatedly} perform \textit{complete} ``sink tube push'' procedures for rational tangle strands, without getting new double points that are not safe or collapsing existing unsafe double points, whenever the rational number $\frac{p}{q}$ with $(p,q)=1$ characterizing the $\beta$-arc has $p,q\geq 2$. Each sink tube push creates a new tail. However, the tail is in the sink bigon of the next sink tube push (if there is a next one), and will be enclosed in loop and collapsed by the next (complete) sink tube push. See Figure~\ref{fig:24} for an example. 

    We remark that the endpoints of the $\beta$-arc change as the rational tangle strand changes. In fact, the endpoints of $\beta$ are determined by the parity types of the rational number: if the rational number is odd/even, then the endpoints of $\beta$-arc are $c$ and $d$, if odd/odd, then the endpoints are $b$ and $d$, and if even/odd the endpoints are $b$ and $c$. We also remark that our sink tube push can be performed regardless of the location of the endpoints, as long as the source bigon is at $a$.

    Since $p+q$ is decreasing, this process of repeated complete sink tube pushes will terminate. It terminates when either $p=1$ or $q=1$. We show that now every double point on $S$ is safe. The remaining $\beta$-arcs on $S$ contain a rational tangle strand connecting two basepoints of the three $\{b,c,d\}$ and a tail from the last sink tube push (notice we are not able to further enclose this in a loop and collapse). However, the double point on this tail is safe (connected to a basepoint), so we still only need to pay attention to the double points from the intersection of $\alpha$ and the $\beta$-rational tangle strand. If $q=1$, then $\alpha$ and $\beta$ only intersect at $b$; but this is impossible, since our anchor $e$ is not moved and serves as another intersection. Hence $p=1$. Now similarly to our argument before (at the beginning of Step 3), the $\beta$-rational tangle strand intersects $\alpha$ always in the same direction, and all double points are safe. It then follows that after our repeated operations we get a branched surface $\mathcal{A'''}$ that has no sink disk.
    ~\\

    We are now in position to apply Lemma~\ref{lem:3}. We claim that no horizontal boundary component of $\mathcal{A'''}$ is a disk. Since our splittings are always along disks and modelled on Figure~\ref{fig:16} middle, they do not change the topology of the horizontal boundary, hence we only need to check for $\mathcal{A''}$. Now consider horizontal boundary components of $\mathcal{A''}$. The boundary of these horizontal boundary components must correspond to the boundary or branch locus of $\mathcal{A''}$. For the boundary circle of $\mathcal{A''}$ corresponding to the $\mathcal{C'}$-cusp, there are two components of horizontal boundary bounded by it, each homeomorphic to $\partial_h^{\pm}$ (they are the sides of $\partial_h^{\pm}$ that $\mathcal{A}$ is not attached to), thus not disks. Other possible boundary circles are $\delta$, the $(\alpha,\beta)$-cusp, and the boundary circle in the sink bigon - exactly the cusps and boundary circles that remain in $\mathcal{B'}$. Hence for a disk horizontal boundary component in $\mathcal{A''}$, we can find a corresponding properly embedded disk in the regular neighborhood of $\mathcal{B'}$. However, $N(\mathcal{B'})$ is homeomorphic to the (1,1)-knot complement, and we recall from proof of Lemma~\ref{lem:11} that all the boundary circles and cusps of $\mathcal{B'}$ correspond to the meridian of the knot. Our claim then follows from the fact that there is no disk bounded by meridian in the knot complement.

    Also since $\mathcal{B'}$ is co-oriented, it cannot carry a sphere or a torus, for they will then be non-separating in the total lens space. Hence $\mathcal{A'''}$ also cannot carry any sphere or torus. Now we can use Lemma~\ref{lem:3} and conclude that $\mathcal{A'''}$ fully carries a lamination. Hence $\mathcal{A''}$ also fully carries a lamination.
\end{proof}

    %If we regard the framing rectangle as the unit square and extend it to the plane by reflections, we can identify our original rational tangle strand as a segment $K$ connecting $(0,0)$ and $(q,p)$. Then by our construction of $\frac{s}{t}$, we see that $(t,s)$ is either of horizontal distance $\frac{1}{p}$ or vertical distance $\frac{1}{q}$ to $K$ (see Figure~\ref{fig:24}.$(f)$). It follows that $|pt-qs|=1$, and thus we have $0\leq \frac{s}{t}\leq 1$.

\section{Co-oriented taut foliations from laminations}
\label{sec:4}

In this section we construct co-oriented taut foliations in non-simple (1,1)-knot complements.

\begin{proof}[Proof of Theorem~\ref{thm}]
    Let $(\Sigma,\alpha,\beta,z,w)$ be some non-simple reduced (1,1)-diagram representing the non-simple (1,1)-knot, and $\mathcal{B}$ be the associated branched surface. Notice that we actually have $N(\mathcal{B})$ homeomorphic to the knot complement $M$ by Lemma~\ref{lem:5}. Now by Proposition~\ref{prop:1}, $\mathcal{B}$ fully carries some lamination $\mathcal{L}\subset N(\mathcal{B})$, and by blowing up leaves we can further assume $\partial_h N(\mathcal{B})\subset \mathcal{L}$. Since our $\mathcal{B}$ is co-oriented by definition, components of $N(\mathcal{B})-\mathcal{L}$ are trivial $I$-bundles over oriented surfaces, and $\mathcal{L}$ extends to a co-oriented foliation $\mathcal{F}$ of $N(\mathcal{B})$.

    \begin{figure}[!ht]
        \begin{overpic}[scale=0.7]{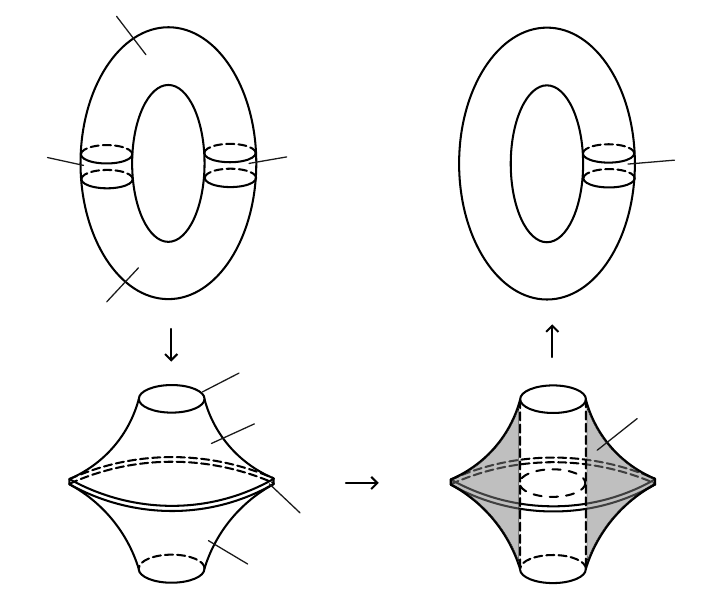}
            \put(-2,61){cusp}
            \put(40,61){boundary}
            \put(12,81){$\tilde{\partial_h^+}$}
            \put(10.5,39.5){$\tilde{\partial_h^-}$}
            \put(34,31.5){boundary}
            \put(36,24){$\tilde{\partial_h^+}$}
            \put(42,10){cusp}
            \put(35,4){$\tilde{\partial_h^-}$}
            \put(89,25){(annulus)$\times I$}
            \put(94,60.5){boundary}
        \end{overpic}
        \caption[Fig 13]{Attaching thickened annulus}
        \label{fig:28}
    \end{figure}
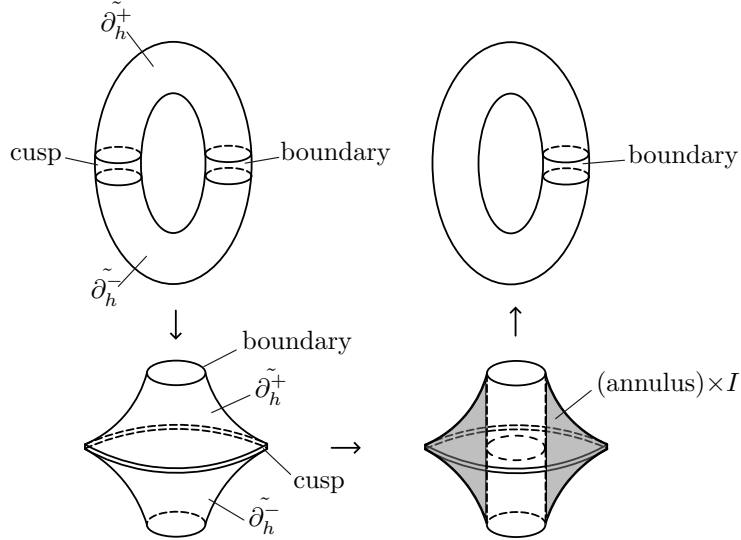

    Foliation $\mathcal{F}$ is tangent to the boundary torus of $N(\mathcal{B})$ at the two annular horizontal boundary components $\tilde{\partial_h^+},\tilde{\partial_h^-}$, and transverse to the vertical boundary of $N(\mathcal{B})$. Moreover, $\mathcal{F}$ intersects the annulus of vertical boundary at the cusp in a product foliation by circles. We can thus extend $\mathcal{F}$ by attaching some (annulus$\times I$) (foliated trivially by the annuli) to the cusp, connecting the two annular horizontal boundary components. (See shaded part of Figure~\ref{fig:28}.) The new foliation $\mathcal{F}'$ still foliates a 3-manifold homeomorphic to $M$, and intersects the boundary torus transversely in a suspension foliation of the meridional slope, with possible nontrivial holonomy at the annulus coming from the circle boundary of $\mathcal{B}$, and product foliation elsewhere.

    It remains to show that $\mathcal{F}'$ is taut (it is co-oriented since $\mathcal{F}$ is). In fact, to each $I$-fiber of $N(\mathcal{B})$ we can take a transverse arc connecting its endpoints in the (annulus$\times I$) attached; thus by joining the $I$-fiber with this arc we get a closed transverse circle. Since each leaf of $\mathcal{F}'$ must pass through some $I$-fiber of $N(\mathcal{B})$, to each leaf there is a transverse circle passing through it, and $\mathcal{F}'$ is hence taut.
\end{proof}
~\\

\bibliography{references.bib}
\bibliographystyle{alpha}

\end{document}